\documentclass[letterpaper,10pt]{ieeeconf}
 \IEEEoverridecommandlockouts
 \title{Training Classifiers For Feedback Control}
 \author{Hasan A. Poonawala, Niklas Lauffer, and Ufuk Topcu
 \thanks{This material is based upon work supported by the National Science Foundation under Grant No. 1646522 and Grant No. 1652113. }
 \thanks{Hasan A. Poonawala is with the Department of Mechanical Engineering,
         University of Kentucky, Lexington, KY 40506, USA.
         {\tt\small hasan.poonawala@uky.edu}}
 \thanks{Niklas Lauffer is with the University of Texas, Austin, TX 78712, USA.
         {\tt\small nlauffer@utexas.edu}}
 \thanks{Ufuk Topcu is with the Department of Aerospace Engineering, University of Texas,
         Austin, TX 78712, USA.
         {\tt\small utopcu@utexas.edu}}%
 } 
 \usepackage{amsmath,amssymb,amsfonts}

\usepackage{amsthm}
\usepackage{breqn}
\usepackage{subcaption}
\usepackage{url}
\usepackage{cite}
\usepackage{tikz}
\usepackage{pgfplots}
\usetikzlibrary{arrows,decorations.pathmorphing,positioning,fit,trees,shapes,shadows,automata,calc,intersections,decorations.markings,pgfplots.fillbetween} 
\usepackage{algorithm} 
\usepackage{algorithmic}

 \usepackage{algorithm}
 \usepackage{algorithmic}
\newcommand{\pd}[2]{\frac{\partial #1}{\partial #2}}

\newcommand{\bmat}[1]{\begin{bmatrix}#1\end{bmatrix}}

\newcommand{\xth}[1]{#1^{\mathrm{th}}}
\newcommand{\state}{x}
\newcommand{\State}{X}

\newcommand{\featuredim}{m}
\newcommand{\measurement}{y}
\newcommand{\Measurement}{Y}
\newcommand{\measuredim}{m}
\newcommand{\control}{u}
\newcommand{\Control}{U}

\newcommand{\clabel}{b}
\newcommand{\labone}{b_1}
\newcommand{\labtwo}{b_2}
\newcommand{\lblSet}{L}

\newcommand{\mode}{i}

\newcommand{\Mode}{I}
\newcommand{\numtrain}{N_D}
\newcommand{\trdata}{D}

\newcommand{\classifier}{C}

\newcommand{\Hsetfn}{\mathcal{H}}
\newcommand{\Hselection}{\hat \Hsetfn}
\newcommand{\cweights}{w}
\newcommand{\cweightsm}{w_1}
\newcommand{\ccon}{\cweights_{0}}

\newcommand{\Penvir}{P}

\newcommand{\deviation}{d}
\newcommand{\localangle}{\psi}
\newcommand{\locstate}{\state}
\newcommand{\exaenvir}{\rho}

\newcommand{\linalphaH}{H}
\newcommand{\linalpconst}{h}
\newcommand{\pState}{X}

\newcommand{\pwlSys}{\Omega}
\newcommand{\slack}{q}

\newcommand{\prtition}{\mathcal{P}}
\newcommand{\constraintset}{\mathcal W}

\newcommand{\qState}{Z}
\newcommand{\cMatrix}{E}
\newcommand{\cVec}{e}
\newcommand{\vMatrix}{F}
\newcommand{\vcVec}{f}
\newcommand{\Vprtition}{\mathcal Q}
\newcommand{\diffinclusion}{\mathcal A}
\newcommand{\Pindex}{\Mode_{\prtition}}
\newcommand{\Vindex}{\Mode_{\Vprtition}}
\newcommand{\Dindex}[1]{\Mode_{\diffinclusion_{#1}}}
 
\newcommand{\lyaplevelset}{S}
\newcommand{\R}{\mathbb{R}}

\theoremstyle{definition}
\newtheorem{defn}{Definition}

\theoremstyle{plain}
\newtheorem{thm}{Theorem}
\newtheorem{lem}[thm]{Lemma}
\newtheorem{prop}[thm]{Proposition}

\theoremstyle{remark}

\begin{document}
\maketitle
\begin{abstract}
One approach for feedback control using high dimensional and rich sensor measurements is to classify the measurement into one out of a finite set of situations, each situation corresponding to a (known) control action. %
This approach computes a control action without estimating the state. %
Such classifiers are typically learned from a finite amount of data using supervised machine learning algorithms. %
We model the closed-loop system resulting from control with feedback from classifier outputs as a piece-wise affine differential inclusion. 
We show how to train a linear classifier based on performance measures related to learning from data \emph{and} the local stability properties of the resulting closed-loop system. %
The training method is based on the projected gradient descent algorithm. %
We demonstrate the advantage of training classifiers using control-theoretic properties on a case study involving navigation using range-based sensors. %
\end{abstract}

\vspace{-2mm}
\section{Introduction}
 
A common situation in robotics involves using information-rich sensors, which provide high dimensional measurements, to control the state of a robot in different environments. 
Example of such sensors include cameras and LIDAR. %
Even though the available measurement is high-dimensional, the robot may often only need to identify the current situation it is in and apply a corresponding control, without explicit knowledge of the state. %
Obstacle avoidance using proximity sensors such as SONAR are an example of this strategy. %
A finite set of controls is often sufficient to achieve safe and stable operation of the robot in that environment, where each control in the set corresponds to one of the specific situations that is known to occur. %


A classifier, trained using supervised learning methods, often performs the identification step. %
Once the measurement has been classified into one of the finite possible situations, the system uses a pre-designed control action associated with the classifier output. %
In many robotic systems such as for mobile robots, human expertise is sufficient to design these control actions. %
We refer to such a feedback control system as a classifier-in-the-loop system. %
Figure~\ref{fig:cilsdia} depicts such a feedback mechanism. %
Several feedback systems in the literature are classifier-in-the-loop systems~\cite{Giusti16,Levine2016}, however the evaluation of their properties are almost always empirical. %
We seek to provide a more rigorous approach to the analysis and synthesis of classifiers used for control purposes. %

Given training data, one can use supervised learning methods~\cite{Alpaydin10,Cortes1995} to design a classifier that assigns one of the finite possible controls to a measurement. %
A common approach to supervised learning involves solution of an optimization problem. %
The objective function typically consists of a loss function that penalizes errors between the classifier's prediction for a measurement and the actual target value associated with that measurement in the dataset. %

A low value of the loss function does not necessarily say anything about the properties of the resulting closed-loop system. %
We require a method to relate the closed-loop system properties with the parameters of the classifier. %
We wish to reformulate existing techniques for training classifiers in a way that is meaningful for their use as feedback controllers. %

An important observation that permits development of the training methods we will present involves the recognition that a classifier-in-the-loop control scheme can be modeled~\cite{PoonawalaCDC17} using switched~\cite{Filippov1988} and/or hybrid system formalisms~\cite{goebel2012hybrid}. %
The classifier parameters dictate the switching (or guard) surfaces of the closed-loop system. %
Training the classifier is equivalent to determining the appropriate switching surface. %
Analysis of switched systems with variable switching surfaces is central to training of classifier-in-the-loop systems. %
Some methods exist to analyze or design such hybrid systems~\cite{Prajna03,Hu08,Trofino09,JohanssonThesis, Cortes2008,Blanchini1995}. %
We will use methods from~\cite{JohanssonThesis, Cortes2008,Blanchini1995}. %

\begin{figure}[tb]
\centering
\scalebox{0.8}{
\begin{tikzpicture}
\tikzstyle{myedge}=[->,semithick,>=stealth',shorten >=1pt, shorten <= 1pt,auto]
\tikzstyle{mynode}=[thick,rectangle,draw,minimum size = 0.5cm,rounded corners = 0.05cm]
\coordinate (p1) at (0,0);
\draw[thick] ($(p1)-(0.5,1)$) coordinate (O)--++(30:1)coordinate (A)--++(90:1)coordinate (B)--++(150:1)coordinate (C)--cycle;
\foreach \y/\t in {0.1/1,0.2/2} {
\draw[thick] ($(C)! \y*2.1 !(O)$)--++(180:1) node[left] {$\control_{\t}$};}
\draw[thick] ($(C)! 0.4*2.1 !(O)$)--++(180:1) node[left] {$\control_{N}$};
\draw[thick] ($(C)! 0.3*2.0 !(O) -(1.15,0)$) node[left] {\small $\vdots$};
\def\ndhp{2}
\node (dummy) at (p1){};
\node[mynode] (fts) at ($(p1)+(1.25*\ndhp,0)$) {$\dot{\state} = f(\state,\control)$ };
\node[mynode] (sensor) at ($(p1)+(2.25*\ndhp,-0.5*\ndhp)$) {Sensor $\Hsetfn$};
\path[myedge,draw] (fts) -|  node {$\state$}  (sensor);
\node[mynode] (ic) at ($(p1)+(1.25*\ndhp,-\ndhp)$) {Classifier $\classifier$};
\path[myedge,draw] (sensor) |-  node {$\measurement$}  (ic);
\path[myedge,draw] (ic) -|  node {$\clabel$}  ($(O)!0.5!(A)$);
\path[myedge,draw] ($(A)!0.5!(B)$) --  node {$\control$}  (fts);
\node[text width = 2.5cm,align=center] at ($(C)! 0.6*2.0 !(O) -(1.25,0)$) {finite control set $\Control$};
\end{tikzpicture}
}
\caption{A dynamical system with a classifier $\classifier$ in the feedback loop. %
The measurement $\measurement$ obtained by the sensor in a state $\state$ 
depends on the (unknown) map $\Hsetfn$. %
The classifier converts the measurement into a label $\clabel$ that determines which control $\control_\mode \in \Control$ is chosen as the control input $\control$.}
\label{fig:cilsdia}
\vspace{-5mm}
\end{figure}
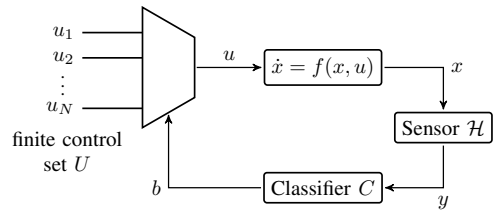
\vspace{-5mm}
\subsection*{Contributions}
This work involves three contributions. %
First, we show how to model the control of dynamical systems via classification using piece-wise affine differential inclusions~\cite{Cortes2008}. %
Second, we formulate the training problem for classifiers used in control as a constrained optimization problem, and derive the corresponding constraints using Lyapunov-based stability conditions appropriate for piece-wise affine differential inclusions~\cite{Cortes2008, JohanssonThesis}. %
These constraints are bilinear in the optimization variables. %
Our third contribution is to develop an algorithm for solving the constrained optimization problem based on the projected gradient descent algorithm. %

The work in this paper differs from~\cite{PoonawalaCDC17} in that here we propose computational algorithms to design classifier-in-the-loop systems. %
We apply our training method for classification-based feedback control to a robot navigation problem simulated in ROS Gazebo. %

\quad  

\section{Control-Oriented Training Of Classifiers }
\label{sec:problem}
\def\cweights{w}
\def\R{\mathbb{R}}
Consider a classifier $C$ parametrized by a set of weights $\cweights \in \R^{s}$ for some $s \in \mathbb{N}$. 
Training a classifier typically involves the solution of an optimization problem in the form 
\begin{dmath}
\min_{\cweights \in \R^{s}} \quad l_{data}(\cweights),
\label{eq:classifier}
\end{dmath}
\noindent where $l_{data} \colon \R^{s} \to \R$ is a loss function for $\cweights$ evaluated on the data set $D$. 

Assuming that we can compute a gradient of $l_{data}(w)$, an iterative algorithm to compute a local optimal point $w^*$ is given by 
\begin{dmath}
w_{k+1} = w_{k} - \alpha_k \nabla l_{data}(w)^T
\end{dmath}
\noindent where $\nabla l_{data}(w)$ is the gradient, $\alpha_k$ is the learning rate, and $w_{k}$ is the estimate of $w^*$ at the $\xth{k}$ iteration. %

We modify~\eqref{eq:classifier} to train classifiers that will be used for control in two ways. %
We add a term $l_{control}(\cweights)$ to the objective function. %
We also add constraints on $\cweights$ such that all feasible solutions of the optimization problem correspond to closed-loop systems with desired behavior. %
Let $\constraintset$ be the feasible set under these constraints. %
The resulting optimization problem that trains classifiers for control is
\begin{dmath}
\min_{w \in \constraintset}  \quad l_{data}(w) + l_{control}(w) \label{eq:classifierintheloop1}	
\end{dmath}

Given the constraints defining $\constraintset$, we can compute an optimal solution $w^{*}$ using projected gradient descent. %
This procedure involves solution of the iterative equations
\begin{IEEEeqnarray}{rl}
w_{k+1}' = &\ w_{k} - \alpha_k  \left(\nabla l_{data}(w) + \nabla l_{control}(w) \right)^T, \textrm{ and \ } \label{eq:intropgd1}\\
w_{k+1} = &\  \arg \min_{w \in \constraintset} \quad \lVert w - w_{k+1}' \rVert. \label{eq:intropgd2}	
\end{IEEEeqnarray}
Figure~\ref{fig:projgraddescent} depicts the procedure. 

The key problems considered in this paper are three-fold. %
First, how do we robustly model the closed-loop system derived from a control scheme involving classification?
Second, how do we derive the constraints on the classifier parameters that, if satisfied, imply desirable closed-loop properties of the modeled system?
Third, given these constraints (equivalently, the constraint set $\constraintset$), how do solve the optimization problem in~\eqref{eq:intropgd2}?

The next three sections provide a solution to each of these problems. %
Briefly, we propose the use of piece-wise affine differential inclusions to model the closed-loop system, the use of piece-wise linear Lyapunov functions to certify closed-loop properties, and algorithms for solving biconvex optimization problems to solve~\eqref{eq:intropgd2}. %
These solutions, together, allow us to train classifiers that yield control performance guarantees on the closed-loop system behavior. %
Note that some of the choices we make to solve each problem are important for being able to combine the three solutions. %

\begin{figure}[tb]
\centering
\begin{tikzpicture}
	\node[inner sep=0pt] (main) at (0,0) {\includegraphics[width=0.2\textwidth]{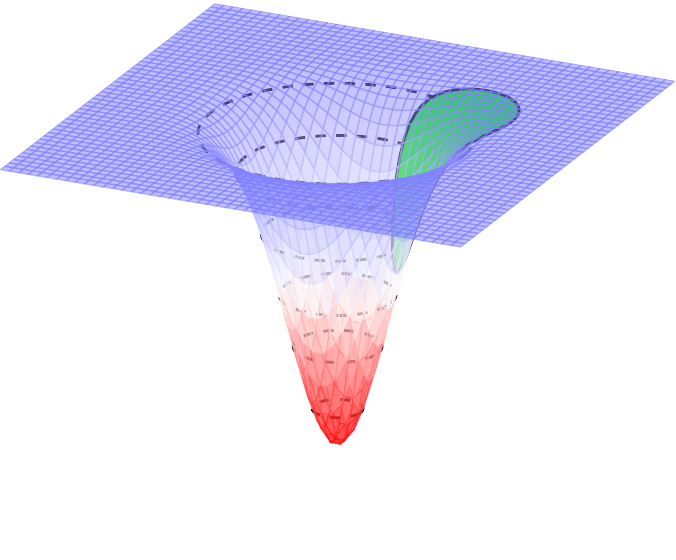}};
	\node[right of=main,node distance = 4cm] (sec) {\includegraphics[width=0.2\textwidth]{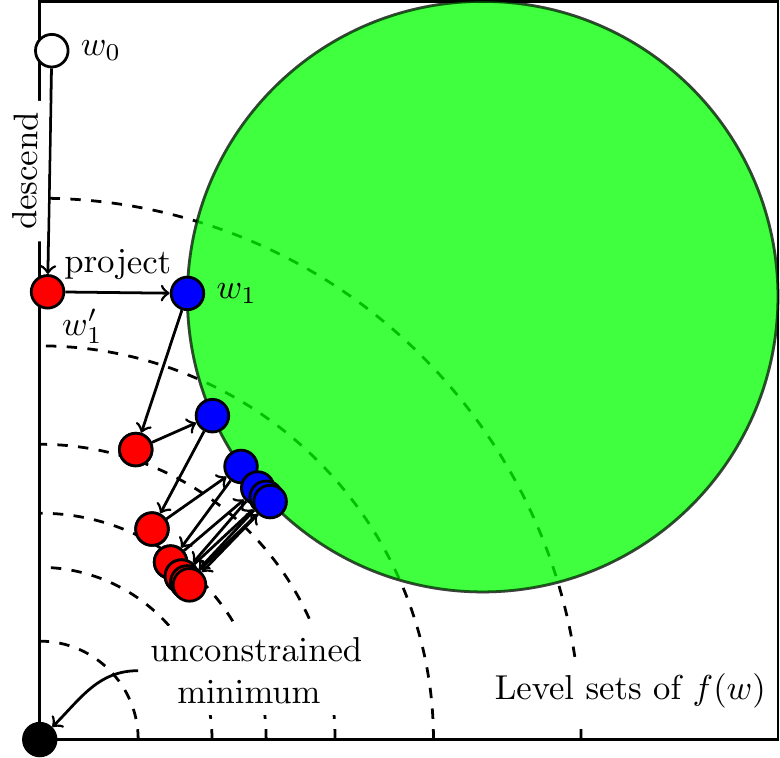}};
	\node[below of=main,node distance = 2cm] {a)};
	\node[below of=sec,node distance = 2cm] {b)};
	\path[draw,->] (-1,-1) -- +(0,1) node[below left] {$l(w)$};
	\path[draw,->] (60:1) to [out = 90, in=180] +(60:1) node[right] {$f(\constraintset)$};
	\node at (4.5,1) {$\constraintset$};
	\path[draw,->] (60:1) to [out = 90, in=180] +(60:1) node[right] {$f(\constraintset)$};
\end{tikzpicture}
\caption{\small The projected gradient descent algorithm. a) The surface denotes the value of the objective function, and the green patch denotes the objective values for points in the feasible region. b) The red and blue points depict the successive iterates of the algorithm. The descent step creates an infeasible point (red), which is projected onto the feasible set to obtain a feasible point (blue). In this case, the feasible iterates (blue dots) converge to the optimal solution.}
\label{fig:projgraddescent}
\vspace{-4mm}
\end{figure}
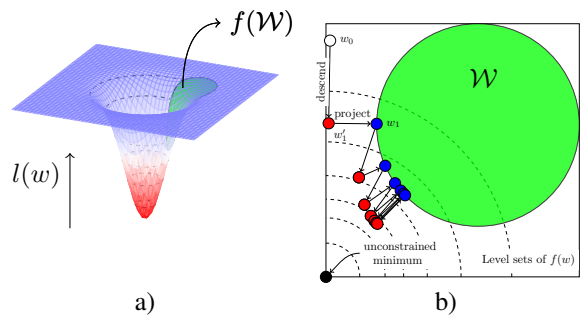

\section{Classifier-in-the-Loop Systems}
\label{sec:cils}

In this section we show how to model the dynamics of a classifier-in-the loop system as a piece-wise affine differential inclusion~\cite{Cortes2008} of the form
\begin{dmath}
	\dot \state (t) \in \diffinclusion(\state(t)), \label{eq:targetpwadi}
\end{dmath}
\noindent where $x \in \R^n$ and $\diffinclusion \colon \R^n \to 2^{\R^n}$ is a set-valued map constructed using affine functions. %
\vspace{-3mm}
\subsection{How Do We Model Classifiers For Control?}
\label{ssec:classifiersmodel}

In general, a classifier $C\colon \Measurement \to \lblSet$ is a map that assigns a unique label $\clabel \in \lblSet$ to an measurement $\measurement \in \Measurement$, where $\Measurement \subseteq \R^{\featuredim}$ is the space of measurements. %
The set $\lblSet$ is typically finite. %
Several methods are available to construct a classifier $\classifier$~\cite{Alpaydin10}. %
We focus our attention on classifiers that partition the measurement space into polytopes of identically classified points. %
This class of classifiers is fairly large, and includes linear classifiers, rectifier networks (common in deep learning), decision trees, and nearest-neighbor classifiers. %

Let $\cweights$ denote the parameters of a classifier $\classifier$. %
We refer to the procedure for obtaining $\cweights$ using data as training the classifier. %
We represent the classifier $\classifier$ as 
\begin{dmath}
\classifier(\measurement) = \clabel_i  \condition{if $\bar E_i(\cweights) \measurement + \bar e_i(\cweights) > 0$},\label{eq:multiclassifier}
\end{dmath}
\noindent where $i \in I$, $I$ is the index set of convex polytopes in the partition induced by the classifier. %
Note that non-convex polytopes are easily divided into convex polytopes. %
 
When $\lblSet = 2$, the classifier~\eqref{eq:multiclassifier} becomes
\begin{dmath}
\classifier(\measurement) = \begin{cases}
                {\labone} & \condition{if $\cweightsm^T \measurement + \ccon > 0$}\\
                {\labtwo} & \condition{if $\cweightsm^T \measurement + \ccon  < 0$},
                \end{cases}	
                \label{eq:basicclassifiertwo}
\end{dmath}
\noindent where $\cweights= (\cweightsm, \ccon) \in \R^{\measuredim+1}$ are learned from data. %
In this case, the classifier parameters $\cweights$ in~\eqref{eq:basicclassifiertwo} directly yield the partition (equivalently, parameters $\bar E_i(\cweights)$ and $\bar e_i(\cweights)$) of $\Measurement$. %
For nearest-neighbor classifiers or decision trees, the parameters $\bar E_i(\cweights)$ and $\bar e_i(\cweights)$ will need to be derived from the trained classifier parameters $\cweights$. %

When $\lvert \lblSet \rvert > 2$, one often constructs a classifier $\classifier \colon \Measurement \to \lblSet$ by combining multiple binary classifiers~\eqref{eq:basicclassifiertwo} in different ways. %
One way is to construct a decision tree, where every node is a binary classifier. %
Another way is to train multiple classifiers parameterized as $\cweights^j = (\cweightsm^j, \ccon^j)$, where each classifier $\cweights^j$ distinguishes between one of the $\binom{\lvert \lblSet \rvert}{2}$ possible pairs of labels from $\lblSet$. %
The partitions induced by this approach can be derived from $\cweights^j$. %
This multi-label classification scheme is known as one-vs-one classification. %
Instead, we can train $\lvert \lblSet \rvert$ classifiers that separate each label from all other labels. %
This classification scheme is known as one-vs-all classification. %
\begin{figure}
\centering
	\begin{tikzpicture}
\tikzstyle{myedge}=[shorten >=1pt,auto,semithick]
\tikzstyle{axes}=[->,opacity = 1.0]
\tikzstyle{mysubcaption}=[text width = 4cm, align = center]

\def \xaxissize {1.7};
\def \yaxissize {1.7};
\def \lowoffset {1.0}
\def \uppoffset {0.5}
\def \figxspace {4.5}
\def \figyspace {4.5}

\coordinate (or1) at (0,0);
\coordinate (lowpoint) at ($(or1)+(\lowoffset, -\yaxissize )$);
\coordinate (uppoint) at ($(or1)+(-\xaxissize,\uppoffset)$);
\coordinate (c1) at ($(or1)+(\xaxissize,\yaxissize)$);
\coordinate (c2) at ($(or1)+(-\xaxissize,\yaxissize)$);
\coordinate (c3) at ($(or1)+(-\xaxissize,-\yaxissize)$);
\coordinate (c4) at ($(or1)+(\xaxissize,-\yaxissize)$);
\path[draw, name path = A] (uppoint) to [out = 0, in = 150] (or1) (or1) to [out = -30, in = 90] (lowpoint);
\path[name path = B] (uppoint) -- (c3) -- (lowpoint);
\path[name path = C] (uppoint) -- (c2) --(c1) -- (c4) -- (lowpoint);
\tikzfillbetween[of=B and A]{red!40!white};
\tikzfillbetween[of=C and A]{blue!40!white};
\path[draw,thick] (uppoint) to [out = 0, in = 150] (or1) (or1) to [out = -30, in = 90] (lowpoint);
\draw[axes] ($(or1)-(0:\xaxissize)$) -- ($(or1)+(0:\xaxissize)$) node [anchor = north east] {$\state_1$};
\draw[axes] ($(or1)-(90:\yaxissize)$) -- ($(or1)+(90:\yaxissize)$) node [anchor = north east] {$\state_2$};
\node[below, mysubcaption] at ($(0,-\yaxissize )+(or1)$){\small a) Nonlinear closed-loop system};
\node at ($(or1) - (1,1)$) {$f(x,u_1)$};
\node at ($(or1) + (1,1)$) {$f(x,u_2)$};
\node[rotate=-20] at ($(or1) + (150:0.7)+(0,0.2)$) {$w^T \Hsetfn(x) = 0$};
\coordinate (or1) at (\figxspace,0);
\coordinate (lowpoint) at ($(or1)+(\lowoffset, -\yaxissize )$);
\coordinate (uppoint) at ($(or1)+(-\xaxissize,\uppoffset )$);
\coordinate (c1) at ($(or1)+(\xaxissize,\yaxissize)$);
\coordinate (c2) at ($(or1)+(-\xaxissize,\yaxissize)$);
\coordinate (c3) at ($(or1)+(-\xaxissize,-\yaxissize)$);
\coordinate (c4) at ($(or1)+(\xaxissize,-\yaxissize)$);
\coordinate (cS) at ($(or1)+(0,-\yaxissize)$);
\coordinate (cN) at ($(or1)+(0,\yaxissize)$);
\path[draw, name path = AA] (uppoint) to [out = 0, in = 150] (or1);
\path[draw, name path = AB] (or1) to [out = -30, in = 90] (lowpoint);
\path[name path = t1] (or1) -- ($(or1)+(150:2)$);
\path[name path = t2] (c3) -- (c2);
\path[name intersections={of= t1 and t2}];
\coordinate (bor1) at (intersection-1);
\path[name path = t3] (or1) -- ($(or1)+(-30:2)$);
\path[name path = t4] (c1) -- (c4) -- (c3);
\path[name intersections={of= t3 and t4}];
\coordinate (bor2) at (intersection-1);
\path[name path = BB] (uppoint) -- (c3) -- (cS) -- (or1);
\path[name path = CC] (or1) -- (cS) -- (lowpoint);
\path[name path = DD] (bor1) -- (c2) -- (cN) -- (or1);
\path[name path = EE] (or1) -- (cN) -- (c1) -- (bor2);
\tikzfillbetween[of=BB and AA]{red!40!white};
\tikzfillbetween[of=CC and AB]{red!40!white};
\tikzfillbetween[of=DD and AA]{blue!40!white};
\tikzfillbetween[of=EE and AB]{blue!40!white};
\fill[purple!70!white] (or1) to (uppoint) -- (bor1) -- (or1);
\fill[purple!70!white] (or1) to (bor2) -- (c4) -- (lowpoint) -- (uppoint) --(bor1);
\path[draw,thick] (uppoint) to [out = 0, in = 150] (or1) (or1) to [out = -30, in = 90] (lowpoint);
\draw[axes] ($(or1)-(0:\xaxissize)$) -- ($(or1)+(0:\xaxissize)$) node [anchor = north east] {$\state_1$};
\draw[axes] ($(or1)-(90:\yaxissize)$) -- ($(or1)+(90:\yaxissize)$) node [anchor = north east] {$\state_2$};
\path[draw] (bor2) -- (c4) -- (lowpoint) -- (uppoint) --(bor1) -- (bor2);
\node[below, mysubcaption] at ($(0,-\yaxissize )+(or1)$){\small b) Over-approximation of Figure~\ref{fig:pwaapproximation}a};
\node at ($(or1) - (1,1)$) {$\diffinclusion_1$};
\node at ($(or1) + (1,1)$) {$\diffinclusion_2$};
\node at ($(or1) + 1.1*(1.1,-1)$) {$\diffinclusion_3$};
\coordinate (or1) at (0,-\figyspace);
\coordinate (lowpoint) at ($(or1)+(\lowoffset, -\yaxissize )$);
\coordinate (uppoint) at ($(or1)+(-\xaxissize,\uppoffset)$);
\coordinate (c1) at ($(or1)+(\xaxissize,\yaxissize)$);
\coordinate (c2) at ($(or1)+(-\xaxissize,\yaxissize)$);
\coordinate (c3) at ($(or1)+(-\xaxissize,-\yaxissize)$);
\coordinate (c4) at ($(or1)+(\xaxissize,-\yaxissize)$);
\path[draw, name path = AA] (uppoint) to [out = 0, in = 150] (or1);
\path[draw, name path = AB] (or1) to [out = -30, in = 90] (lowpoint);
\path[name path = t1] (or1) -- ($(or1)+(150:2)$);
\path[name path = t2] (c3) -- (c2);
\path[name intersections={of= t1 and t2}];
\coordinate (bor1) at (intersection-1);
\path[name path = t3] (or1) -- ($(or1)+(-30:2)$);
\path[name path = t4] (c1) -- (c4) -- (c3);
\path[name intersections={of= t3 and t4}];
\coordinate (bor2) at (intersection-1);
\fill[blue!40!white] (bor1) to (c2) -- (c1) -- (bor2);
\fill[red!40!white] (bor2) -- (c4) -- (c3) -- (bor1);
\path[draw,thick,dashed] (uppoint) to [out = 0, in = 150] (or1) (or1) to [out = -30, in = 90] (lowpoint);
\draw[axes] ($(or1)-(0:\xaxissize)$) -- ($(or1)+(0:\xaxissize)$) node [anchor = south east] {$\state_1$};
\draw[axes] ($(or1)-(90:\yaxissize)$) -- ($(or1)+(90:\yaxissize)$) node [anchor = north east] {$\state_2$};
\path[draw,thick] (bor2) -- (bor1);
\node[below, mysubcaption] at ($(0,-\yaxissize )+(or1)$){\small c) Linearization of Surface};
\node at ($(or1) - (1,1)$) {$\diffinclusion_1$};
\node at ($(or1) + (1,1)$) {$\diffinclusion_2$};
\node[rotate=-30] at ($(or1) + (60:0.3)+(150:0.2)$) {$w^T ( Hx + h) = 0$};
\coordinate (or1) at (\figxspace,-\figyspace);
\coordinate (lowpoint) at ($(or1)+(\lowoffset, -\yaxissize )$);
\coordinate (uppoint) at ($(or1)+(-\xaxissize,\uppoffset )$);
\coordinate (c1) at ($(or1)+(\xaxissize,\yaxissize)$);
\coordinate (c2) at ($(or1)+(-\xaxissize,\yaxissize)$);
\coordinate (c3) at ($(or1)+(-\xaxissize,-\yaxissize)$);
\coordinate (c4) at ($(or1)+(\xaxissize,-\yaxissize)$);
\coordinate (cS) at ($(or1)+(0,-\yaxissize)$);
\coordinate (cN) at ($(or1)+(0,\yaxissize)$);
\path[draw, name path = AA] (uppoint) to [out = 0, in = 150] (or1);
\path[draw, name path = AB] (or1) to [out = -30, in = 90] (lowpoint);
\path[name path = t1] (or1) -- ($(or1)+(150:2)$);
\path[name path = t2] (c3) -- (c2);
\path[name intersections={of= t1 and t2}];
\coordinate (bor1) at (intersection-1);
\path[name path = t3] (or1) -- ($(or1)+(-30:2)$);
\path[name path = t4] (c1) -- (c4) -- (c3);
\path[name intersections={of= t3 and t4}];
\coordinate (bor2) at (intersection-1);
\path[name path = BB] (uppoint) -- (c3) -- (cS) -- (or1);
\path[name path = CC] (or1) -- (cS) -- (lowpoint);
\path[name path = DD] (bor1) -- (c2) -- (cN) -- (or1);
\path[name path = EE] (or1) -- (cN) -- (c1) -- (bor2);
\tikzfillbetween[of=BB and AA]{red!40!white};
\tikzfillbetween[of=CC and AB]{red!40!white};
\tikzfillbetween[of=DD and AA]{blue!40!white};
\tikzfillbetween[of=EE and AB]{blue!40!white};
\fill[purple!70!white] (or1) to (uppoint) -- (bor1) -- (or1);
\fill[purple!70!white] (or1) to (bor2) -- (c4) -- (lowpoint) -- (or1);
\path[draw,thick] (uppoint) to [out = 0, in = 150] (or1) (or1) to [out = -30, in = 90] (lowpoint);
\draw[axes] ($(or1)-(0:\xaxissize)$) -- ($(or1)+(0:\xaxissize)$) node [anchor = north east] {$\state_1$};
\draw[axes] ($(or1)-(90:\yaxissize)$) -- ($(or1)+(90:\yaxissize)$) node [anchor = north east] {$\state_2$};
\path[draw] (or1) -- (bor2) (lowpoint) -- (or1) -- (bor1) (uppoint) -- (or1);
\node[below, mysubcaption] at ($(0,-\yaxissize )+(or1)$){\small d) Over-approximation of Figure~\ref{fig:pwaapproximation}c};
\node at ($(or1) - (1,1)$) {$\diffinclusion_1$};
\node at ($(or1) + (1,1)$) {$\diffinclusion_2$};
\node at ($(or1) + 1.1*(1.1,-1)$) {$\diffinclusion_3$};
\node at ($(or1) + (-1.5,0.7)$) {$\diffinclusion_3$};
\end{tikzpicture}
	\vspace{-5mm}
	\caption{\small [Figure best viewed in color.] Sketch of the modeling steps that lead to a piece-wise affine differential inclusion model, such as in a). %
	a) The (binary) classifier creates a switched dynamical system with a nonlinear switching surface. %
	b) We can define partitions and differential inclusions to over-approximate the original system with differential inclusions $\diffinclusion_{(\cdot)}$. %
	c) To obtain partition parameters that depend linearly on $\cweights$, we can linearize $\Hsetfn$. %
	d) Again, we over-approximate the system to account for uncertainties. }
	\label{fig:pwaapproximation}
	\vspace{-4mm}
\end{figure}
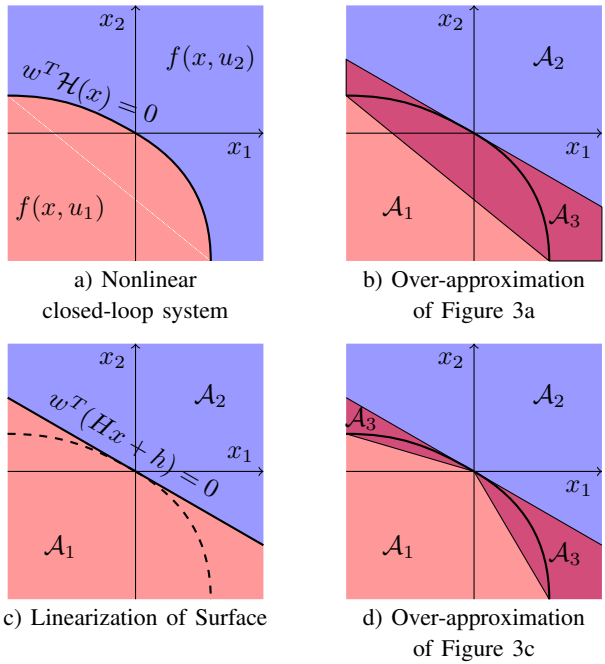
\vspace{-2mm}
\subsection{What Data Are Required?}
The set $\trdata$ of training data generally consists of $\numtrain$ triples $(\state^k,\measurement^k,\clabel^k)$ where $\state$ is the state of the robot, $\measurement$ is the measurement obtained in that state, $\clabel$ is the class label associated with the measurement, and $k$ denotes the index of the triple in the dataset. %

We assume that some labels $\clabel \in \lblSet$ correspond to known control actions $\control_i \in \Control$ that should be taken when the corresponding measurement is observed, according to the data. %
The classifier predicts labels for each measurement, effectively predicting a control action for each measurement. %

To analyze the classifier, we also learn an approximate map $\measurement = \Hselection(\state)$. %
The approximation $\Hselection$ may be used for all states $\state \in \State$, or may be a local approximation using data corresponding to some neighborhood. %
\vspace{-2mm}
\subsection{How Do We Model The Control?}
When we associate a specific control $\control_i \in \Control$ action to each label $\clabel_i \in \lblSet$, the classifier $C$ effectively assigns a control $\control_{i} \in \Control$ to a measurement, where $i \in \{1,2,\dots,\lvert\Control\rvert \}$ (see Figure~\ref{fig:cilsdia}). %
That is,
\begin{dmath}
\control(t) = C( \measurement(t)). \label{eq:controlfrommeas}
\end{dmath}
Recall that $\Control$ is finite and $u(t)$ switches in time between its elements. %
Note that we obtain the control input from the measurement, \emph{not} the state. %

We approximate each vector field $f(\state,\control_i)$ using an affine differential inclusion $\diffinclusion_i $, given by  
\begin{dmath}
\diffinclusion_i =  co \left( \{A_{ik} \state + a_{ik}\}_{k \in \Dindex{i}} \right),
\label{eq:localdynamicsapprox}
\end{dmath}
\noindent where $co(\cdot)$ is the convex hull operation, and $\Dindex{i}$ is the (finite) index set of the affine vector fields that define $\diffinclusion_i$. %
For $\diffinclusion_i$ to be a valid over-approximation of $f(x,u_i)$ over some set $S \subseteq \R^n$, we require that  $ \cup_{x \in S} f(x,u_i) \subseteq \diffinclusion_i$. %
\vspace{-2mm}
\subsection{How Do We Derive the Closed-Loop Dynamics?}
To derive a closed-loop model of the form~\eqref{eq:targetpwadi}, we must be able to model which labels may be assigned in a state $\state$. %
We assume that there is a continuous functional relationship between the measurement and the state, at least locally in space and time, so that $\measurement = \Hsetfn(\state)$. %
This assumption is reasonable for robots operating in slowly changing environments. %
We therefore model~\eqref{eq:controlfrommeas} as 
\begin{dmath}
\control(t) = C \left( \Hsetfn(\state(t)) \right). \label{eq:controlfromstate}
\end{dmath}
We will use an approximation $\Hselection \colon \State \to \Measurement$ of the map $\Hsetfn$, learned from data, to define the dynamics in a state $\state$. %

A key idea of our work is that classifier $\classifier$ partitions the measurement space $\Measurement$, so that control~\eqref{eq:controlfromstate} is effectively a state-based switching control. %
The dynamics switches between the vector fields $f(\state,\control_i)$, where $i \in \{1,\dots, \lvert U \rvert\}$, which we approximate by the inclusions $\diffinclusion_i$. %
From equation~\eqref{eq:multiclassifier}, the partitions of $\R^n$ induced by the partitions of $\Measurement$ are given by inequalities of the form $\bar E_i(\cweights) \Hselection(\state)  + \bar e_i(\cweights) > 0$.
These inequalities may define cells with nonlinear boundaries. %

To model the closed-loop system in a way that is robust to the uncertainty in the switching surfaces, we define convex polytopic partitions where the dynamics may be a combination of the differential inclusions $\diffinclusion_i$ that approximate the vector fields $f(\state,\control_i)$. %
Figure~\ref{fig:pwaapproximation} depicts an example of this process.
 Figure~\ref{fig:pwaapproximation}a shows the case when the map $\Hsetfn$ is nonlinear, so that the switching surface is not planar. %

There are two approaches to obtaining convex partitions. %
The first one involves linearization of $\Hselection$ followed by over-approximation, as in Figures~\ref{fig:pwaapproximation}c and~\ref{fig:pwaapproximation}d. %
The linear estimate in a neighborhood of a point $\state_e$ is given by
\begin{dmath}
\measurement = \Hselection(\state)\\
= \Hselection(\state_e) + \pd{\Hselection}{x}  ( \state - \state_e) + O((x - x_e)^2) \\
 \approx \linalphaH \state + \linalpconst. 
 \label{eq:linearapproxhsetfn}
 \end{dmath}
We locally approximate the partitions by the inequalities 
\begin{IEEEeqnarray}{rl}
& \bar E_i(\cweights) \Hsetfn(\state) +\bar e_i(\cweights) \geq 0 \nonumber \\
\approx & \bar E_i(\cweights) ( H \state + h) +\bar e_i(\cweights) \geq 0 \nonumber \\
= & E_i(\cweights) \state + e_i(\cweights) \geq 0~\label{eq:linearpartitions2}.
\end{IEEEeqnarray}
The set $\diffinclusion_3$ in Figures~\ref{fig:pwaapproximation}b and~\ref{fig:pwaapproximation}d is given by $\diffinclusion_3 = co(\diffinclusion_1 \cup \diffinclusion_2 )$. %
The advantage of this approach is that when using binary classifiers of the form~\eqref{eq:basicclassifiertwo}, the partition parameters $E_i(\cweights)$ and $e_i(\cweights)$ are linear in $\cweights$. %

A second approach is to over-approximate the nonlinear boundary in Figure~\ref{fig:pwaapproximation}a using polyhedral sets, as in Figure~\ref{fig:pwaapproximation}b. %
While this approach is intuitively more appealing than one involving linearization, it is harder to express the partition parameters $E_i(\cweights)$ and $e_i(\cweights)$ as linear functions of $\cweights$. %

By combining the classifier representation~\eqref{eq:multiclassifier}, the differential inclusion dynamics~\eqref{eq:localdynamicsapprox}, and the approximations~\eqref{eq:linearpartitions2}, we model the classifier-in-the-loop system as
\begin{dmath}
\dot{\state} \in \diffinclusion_i  \condition{if $  E_i(\cweights) \state + e_i(\cweights) \geq 0$},
\label{eq:pwadi}
\end{dmath}
\noindent which is a differential inclusion of the form $\dot{x} \in \diffinclusion(\state)$. %
\vspace{-2mm}
\subsection{Is This Model Robust To Uncertainties?}
The use of over-approximations inherently provides robustness to uncertainties and modeling errors. %
The caveat to this approach is that the over-approximation may exhibit far too many trajectories. 
Some of these trajectories may not satisfy the closed-loop properties we wish to certify for the system, while a tighter over-approximation may satisfy the control properties. %
Procedures to obtain such tight over-approximations are beyond the scope of this paper. %

\section{Control-Oriented Constraints On Classifier Parameters}
\label{sec:stability}

Sufficient conditions for certifying the closed-loop properties of~\eqref{eq:pwadi} become conditions on the classifier parameters $\cweights$, that is, they define $\constraintset$ in~\eqref{eq:intropgd2}. %
The properties of interest to us include practical asymptotic stability (ultimate boundedness), forward set invariance, and asymptotic stability. %
These properties physically correspond to low set-point or tracking errors, or to safety via boundedness of the state. %
We want these properties to hold for all closed-loop trajectories. %

Our approach follows much of the existing work on analysis of piece-wise affine differential inclusions~\cite{Cortes2008,Blanchini1995,JohanssonThesis}. %
These papers derive sufficient conditions under different assumptions and parameterizations of the certificates (typically Lyapunov functions) for control properties. %
We present the conditions in terms of our chosen parameterization below. %
We choose polyhedral Lyapunov functions as motivated by~\cite{Blanchini1995}, the stability conditions come from results in~\cite{Cortes2008}, and methods to remove quantifiers from these conditions are inspired by~\cite{JohanssonThesis}. %

A partition $\prtition$ in $\R^n$ is a collection of subsets $\{\pState_i \}_{i \in \Pindex}$; where $\Pindex$ is an index set, $n \in \mathbb{N}$, $\pState_i \subseteq \R^n$ 
for each $i \in \Pindex$, and $Int(\pState_i) \cap  Int(\pState_j) = \emptyset$ for each pair $i,j \in \Pindex$ such that $i \neq j$. %
We define the domain $Dom(\prtition)$ of the partition as $Dom(\prtition) = \cup_{i \in \Pindex} \pState_i$. %
We also refer to the subsets $\pState_i$ in $\prtition$ as the cells of the partition. %
Note that this definition allows some cells in $\prtition$ to represent the boundary between other cells in $\prtition$, which is useful for handling sliding modes. %

A piece-wise affine dynamical system $\pwlSys_\prtition$ associated with partition $\prtition = \{\pState_j \}_{j \in \Pindex}$ is a collection,
\begin{dmath}
	\Omega_\prtition = \left\{ \diffinclusion_i \right\}_{i \in \Pindex}
\end{dmath}
that to each cell $\pState_i \in \prtition$ assigns the affine differential inclusion $\diffinclusion_i = \{A_{ik} \state + a_{ik} \}_{k \in \Dindex{i}}$. %
\begin{dmath}
{\quad \dot \state(t) = co(\diffinclusion_i), \textrm{ if } \state_i(t) \in \State_i}.
\end{dmath}
The cell $\pState_i \in \prtition$ is given by
\begin{dmath}
\pState_i = \{\state \hiderel{\in} \R^n \colon \cMatrix_i \state + 	\cVec_i \hiderel{\geq} 0 \}.
\end{dmath}

We parameterize a continuous polyhedral Lyapunov function $V_{\Vprtition}(\state)$ with a partition $\Vprtition = \{\qState_j \}_{j \in \Vindex}$ and a collection of vectors $\{p_i \}_{i \in \Vindex}$ such that $V_{\Vprtition}(\state) = p_i^T x, \textrm{ if } x \in \qState_i \subseteq \R^n$. %
Each set $\qState_j \in \Vprtition$ is given by $\qState_j = \{x \in \R^n \colon {\vMatrix_j \state \geq 0} \}$, and we assume that $\qState_j $ is pointed at the origin. %

We define index sets that denote the relationship between the system $\pwlSys_\prtition$ and the Lyapunov function $V_{\Vprtition}$. %
Let $\Mode_{cont} \subseteq \Vindex \times \Vindex$ be the set of pairs of indices such that $\qState_i \cap \qState_j \neq \emptyset$. %
Let $\Mode_{dec}$ be the set of all triples $(i,j,k)$ such that $i \in \Pindex$, $k \in \Dindex{i}$, $j \in \Vindex$, and $\pState_i \cap \qState_j \hiderel{\neq} \emptyset$. %

Sufficient conditions on a piece-wise differential inclusion and candidate Lyapunov function that certify the existence of the control properties under consideration are given in~\cite{Cortes2008}. %
The result below formally states these conditions in terms of our parametrization and notation. %

\begin{lem}
\label{lem:mainconstraintlemma}
Let $\pwlSys_{\prtition}$ be a piece-wise affine dynamical system and $V_{\Vprtition}$ be a candidate Lyapunov function. %
Let $\Pindex$, $\Vindex$, $\Mode_{cont}$, and $\Mode_{dec}$ be the index sets associated with $\pwlSys_{\prtition}$ and $V_{\Vprtition}$. %
Let $Dom(\prtition)$ be connected, and let $co(Dom(\prtition))$ contain the origin. %
Let $\lyaplevelset_{max}$ and $\lyaplevelset_{min}$ be the largest and smallest level set of $V_{\Vprtition}(x)$ in $Dom(\prtition)$. %

If the set of constraints
\begin{IEEEeqnarray}{rl}
p_i \hiderel{=} \vMatrix_i^T \mu_i,  &\quad \forall i \in  \Vindex,  \label{eq:lemconstraintsfirst} \\
\mu_i \geq \mathbf{1}, &\quad \forall i \in  \Vindex, \label{eq:lemconstraintssecond}\\
\bmat{\cMatrix_i & \cVec_i \\ 0 & 1}^T v_{ijk} = - \bmat{A_{ik}^T \\ a_{ik}^T} p_j,  &\quad \forall (i,j,k) \in \Mode_{dec}, \label{eq:lemconstraintsdecrease}\\
v_{ijk} \geq \mathbf{1}, &\quad \forall (i,j,k) \in  \Mode_{dec}, \label{eq:lemconstraintsdecrease2}\\
p_i - p_{j}  \hiderel{=} \lambda_{ij} \vcVec_{ij},  &\quad \forall (i,j) \in  \Mode_{cont}, \textrm{ and} \IEEEeqnarraynumspace \label{eq:lemconstraintscont1}\\
\lambda_{ij} \geq 1, &\quad \forall (i,j) \in  \Mode_{cont},  \label{eq:lemconstraintslast}
\end{IEEEeqnarray}
\noindent is feasible, then
\begin{enumerate}
	\item $\lyaplevelset_{max}$ is invariant
	\item $\lyaplevelset_{min}$ is ultimately bounded
\end{enumerate}
Furthermore, if $0 \in Dom(\prtition)$, then the origin of $\pwlSys_{\prtition}$ is asymptotically stable with region of attraction $\lyaplevelset_{max}$. %
\end{lem}
\begin{proof}
See the appendix.
\end{proof}

\section{Control-Oriented Training Using Projected Gradient Descent}
\label{sec:stabalg}	

In this section, we present an algorithm to solve~\eqref{eq:intropgd2}, given a representation of (a subset of) the set $\constraintset$ in terms of~\eqref{eq:lemconstraintsfirst}-\eqref{eq:lemconstraintslast}. %
The constraints~\eqref{eq:lemconstraintsfirst}-\eqref{eq:lemconstraintslast} may be infeasible, since the candidate Lyapunov function (which always exists for suitable partition $\Vprtition$) may not decrease along the dynamics of $\pwlSys_\prtition$. %
To remedy this issue of infeasibility, we relax constraint~\eqref{eq:lemconstraintsdecrease}. %
To account for this relaxation, we modify the objective function. %

The following optimization problem implements~\eqref{eq:intropgd2}:
\begin{IEEEeqnarray}{Cl}
\min_{\cweights, p_i, u_i, v_{ijk}, \lambda_{ij}} & 	\beta \lVert w - w_{k+1}'\rVert_2 +\sum_{(i,j,k) \in I_{dec}} \lVert \slack_{ijk} \rVert_2 \label{opt:lyaprelaxobj}\\
\textrm{s.t.} & \  \nonumber
\end{IEEEeqnarray}
\begin{IEEEeqnarray}{rl}
p_i \hiderel{=} \vMatrix_i^T \mu_i,  &\quad \forall i \in  \Vindex,  \label{opt:lyaprelaxfirst}\\
\mu_i \geq \mathbf{1}, &\quad  \forall i \in  \Vindex, \label{opt:lyaprelaxsecond}\\
\IEEEeqnarraymulticol{2}{l}{
\slack_{ijk} = \bmat{\cMatrix_i(\cweights) & \cVec_i(\cweights) \\ 0 & 1}^T v_{ijk} +\bmat{A_{ik}^T \\ a_{ik}^T} p_j,
} \nonumber\\
  &\quad  \forall (i,j,k) \in \Mode_{dec}, \label{opt:lyaprelaxdec1}\\
v_{ijk} \geq \mathbf{1}, &\quad  \forall (i,j,k) \in  \Mode_{dec},\label{opt:lyaprelaxdec2}\\
p_i - p_{j}  \hiderel{=} \lambda_{ij} \vcVec_{ij},  &\quad  \forall (i,j) \in  \Mode_{cont},\label{opt:lyaprelaxcont1}\\
\lambda_{ij} \geq 1, &\quad  \forall (i,j) \in  \Mode_{cont},  \label{opt:lyaprelaxfinal}
\end{IEEEeqnarray}
\noindent where $\slack_{ijk}$ serve as slack variables that relax the equality constraint~\eqref{eq:lemconstraintsdecrease}, and $\beta \in \R, \beta > 0$ is a weighting factor. %
The optimization variables include $\cweights$, $p_j\ \forall j \in \Vindex$, $v_{ijk}\ \forall (i,j,k) \in I_{dec}$, and $u_i \ \forall i \in \Vindex$. 
The optimization problem~\eqref{opt:lyaprelaxobj}-\eqref{opt:lyaprelaxfinal} has the following property. %
\begin{algorithm}[tb]
    \caption{Projection via Alternate Convex Search}
    \label{alg:projection}
    \begin{algorithmic}
    \REQUIRE  $w_{k+1}'$ from~\eqref{eq:intropgd1}, $\epsilon$, $\beta$ 
    \ENSURE $w_{k+1} \in \constraintset$
    \STATE{$\Delta \gets \infty$}
    \STATE{$ l \gets 0$} \COMMENT{Loop counter}
    \STATE{$w(l) \gets w'_{k+1}$}
    \WHILE{$\Delta  > \epsilon$} 
    \STATE{Compute $\pwlSys_\prtition(l)$ and $V_{\Vprtition}(l)$ using $w(l)$}
    \STATE{$p_j^*, \mu_i^*, \lambda_{ij}^*, v_{ijk}^* \gets$ Solve~\eqref{opt:lyaprelaxobj}-\eqref{opt:lyaprelaxfinal} with $w(l)$ fixed}
    \STATE{$w(l+1) \gets$ Solve~\eqref{opt:lyaprelaxobj}-\eqref{opt:lyaprelaxfinal} with $\mu_i$, $v_{ijk}$, $\lambda_{ij}$ fixed to $\mu_i^*$, $v_{ijk}^*$, and $\lambda_{ij}^*$ respectively}
	\STATE{$\Delta \gets  \lVert w(l+1) - w(l) \rVert$}
	\STATE{$l \gets l+1 $}
	\ENDWHILE
	\STATE{$w_{k+1} \gets w(l)$}
    \RETURN $w_{k+1}$
    \end{algorithmic}
\end{algorithm}
\begin{prop}
\label{prop:feasible}
Let $\pwlSys_{\prtition}$ be a piece-wise affine differential inclusion and $V_{\Vprtition}(\state)$ be a candidate polyhedral Lyapunov function. %
	The optimization problem~\eqref{opt:lyaprelaxobj}-\eqref{opt:lyaprelaxfinal} is feasible. %
\end{prop}
\begin{proof}
The partition $\Vprtition$ allows selection of vectors $\{p_j \}_{j \in \Vindex}$ such that $V_{\Vprtition}(\state)$ is continuous and $V_{\Vprtition}(\state) > 0$ for all $x \neq 0$ by construction. %
This property of $V_{\Vprtition}(\state)$ implies that constraints~\eqref{opt:lyaprelaxfirst}, \eqref{opt:lyaprelaxsecond}, \eqref{opt:lyaprelaxcont1}, and~\eqref{opt:lyaprelaxfinal} are feasible. %
Since $q_{ijk}$ is unconstrainted, \eqref{opt:lyaprelaxdec1} and~\eqref{opt:lyaprelaxdec2} are feasible, independent of the values of the remaining optimization variables. %
\end{proof}

The constraints~\eqref{opt:lyaprelaxfirst}-\eqref{opt:lyaprelaxfinal} are bilinear in the variables of the optimization problem. %
Optimization problems with bilinear constraints are typically NP-hard. %
We use a variant of Alternate Convex Search (ACS)~\cite{Gorski2007} to solve this optimization problem, given in Algorithm~\ref{alg:projection}. %

We begin with the classifier parameters $w'(k+1)$ obtained after a gradient descent step~\eqref{eq:intropgd1}. %
We alternate between solving two convex optimization problems obtained by fixing a  different subset of the variables in~\eqref{opt:lyaprelaxobj}-\eqref{opt:lyaprelaxfinal}. %
We use a value of $\beta \ll 1$ to obtain solutions where the slack variables are zero. %
The convexity and feasibility 
(Proposition~\ref{prop:feasible}) of these problems imply that solutions exist at every iteration. %

The first step of an iteration fixes $w$, and obtain a solution to~\eqref{opt:lyaprelaxobj}-\eqref{opt:lyaprelaxfinal} denoted by $(p_j^*,u_i^*, v_{ijk}^*, \lambda_{ij}^*)$. %
We then solve~\eqref{opt:lyaprelaxobj}-\eqref{opt:lyaprelaxfinal} again, however $\cweights$ is now a variable, and variables $p_j$ remain variable. %
The variables $u_i$,  $v_{ijk}$, and $\lambda_{ij}$ are fixed to the corresponding values $u_i^*$, $v_{ijk}^*$, and $\lambda_{ij}^*$. %
The optimal solution is $(p_j^{**},w^*)$, and $w^*$ is used as the fixed value of $w$ in the next iteration of this procedure. %

Note that the set $\Mode_{dec}$ may change as the partition $\prtition$ changes. %
One way to avoid needing to recompute $\Mode_{dec}$ at every iteration of the Alternate Convex Search is to set $\prtition = \Vprtition$. %
This approach is implicitly taken in~\cite{JohanssonThesis,Blanchini1995}. %


\vspace{-3mm}
\section{Case Study: Path Following}
\label{sec:examplepath}
In our case study, we task a quadrotor equipped with an infra-red-based-scanning device to navigate a canyon-like terrain. %
We use the Gazebo robot simulation environment (see Figure~\ref{fig:valshot}), running on the Robot Operating System, to simulate this scenario. %
We demonstrate that the use of control constraints while training classifiers safeguards against unstable behavior. %
\vspace{-4mm}
\subsection{Modeling}
\label{ssec:modeling}
We model the quadrotor kinematics as a differential-drive like mobile robot. %
That is, we command the quadrotor to achieve has a forward velocity $v$ and an angular velocity $\omega$. %
The simulated quadrotor, however, possesses full inertial and rotational dynamics and implements lower-level controllers to track the commanded velocities, allowing us to abstract away those full dynamics. %

The corridor/canyon defines a path in the plane. %
We can attach a moving coordinate frame, known as a Frenet-Serret frame, to this path, and express the dynamics of the robot within this frame (see Figure~\ref{fig:ddwmr}). %
The configuration of the agent in the Frenet-Serret frame is $\locstate = (\localangle, \deviation)$, where angle $\localangle$ is the heading of the robot with respect to the path-aligned axis of the frame, and offset $\deviation$ is the distance between the robot's location and the origin of the frame (which lies on the path). %
The origin $\locstate = 0$ corresponds to the robot being on the path with its heading aligned with the path tangent. 

The robot uses three control inputs: $\control_1=\bmat{v^*&0}^T$, $\control_2=\bmat{0&\omega^*}^T$, and $\control_3=\bmat{0&-\omega^*}^T$, where $v^* >0$ and $\omega^* > 0$ are constants, so that $U =\lblSet = \{\control_1,\control_2,\control_3 \}$. %
These vectors correspond to moving forward, turning left, and turning right respectively. %
The dynamics under each constant input $\control_\mode \in \Control$ in the local Frenet-Serret frame are given by
\begin{dmath*}
f(\state, \control_1) =\bmat{ \frac{v^*  \exaenvir \cos(\localangle)}{1 - \exaenvir \deviation }  \\  v^* \sin(\localangle) },\ f(\state, \control_2) \hiderel{=} 	\bmat{ \omega^*\\ 0},\textrm{ and } f(\state, \control_3) \hiderel{=} 	\bmat{ -\omega^*\\ 0},
\end{dmath*}
\noindent where $\exaenvir$ is the (unknown) local curvature of the path. %

We approximate the nonlinear dynamics $f(x,u_1)$ by the affine set-valued map $\diffinclusion_1$ given by
\begin{equation*}
\diffinclusion_1 =   {co}_{{\rho \in P}} \left( \bmat{0 & -\exaenvir v^{*}\\v^{*} & 0 } \locstate +  \bmat{v^{*} \exaenvir \\0} \right),
\end{equation*}
\noindent where $\Penvir \subset \R$ is a closed compact set that captures the variation in curvature considered for analysis. %
The dynamics $f(x,u_2)$ and $f(x,u_3)$ are affine and single-valued, so that
\begin{dmath*}
{ \diffinclusion_2 = f(x,u_2) \textrm{ and } \diffinclusion_3 = f(x,u_3).	}
\end{dmath*}
\vspace{-4mm}
\subsection{Training Data and Classification}
The training data $\trdata$ consists of triples $(\locstate^k,\measurement^k,\clabel^k)$, where $x^k = (\localangle^k, \deviation^k)$, $\deviation^k \in \{0.5\textrm{~m}, 0\textrm{~m}, -0.5\textrm{~m}\}$ and $\localangle^k \in \{\pi/6\textrm{~rad}, 0\textrm{~rad}, -\pi/6\textrm{~rad}  \}$. %
The measurement $\measurement$ is a vector of dimension $420$. %
The data points $\state^k$ for which $\deviation^k = 0$ and $\localangle^k$ is $\pi/6\textrm{~rad}$, $0\textrm{~rad}$, or $-\pi/6\textrm{~rad}$ are labeled as $\control_3$, $\control_1$, and $\control_2$ respectively. 
We collect this data in a path that has zero curvature. 
The entire data set is used to estimate $\Hselection$, using polynomial regression. %
We take $v^*$ and $\omega^*$ to be $0.5\textrm{~m/s}$ and $0.15\textrm{~rad/s}$ respectively. %
\begin{figure}[tb]
\centering
\includegraphics[width = 0.37\textwidth]{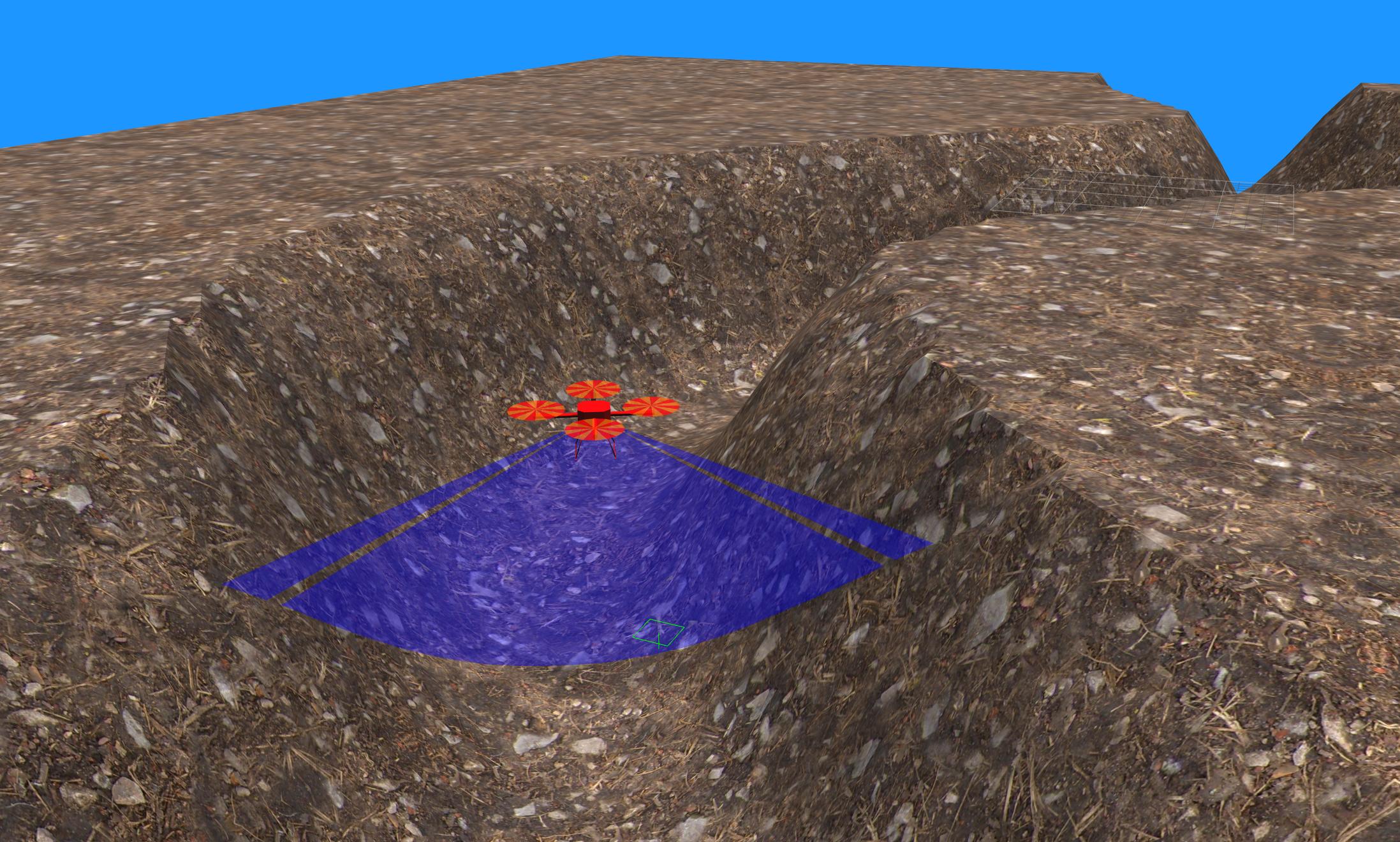}
\caption{\small [Figure best viewed in color.] Screenshot of a simulation in ROS Gazebo of a quadrotor (in red) navigating a simulated terrain. The quadrotor uses a laser scanner to sense the terrain and navigates by classifying measurements, without maintaining a state estimate. The field-of-view of the sensor is depicted by the dark blue region.}
\label{fig:valshot}
\vspace{-3mm}
\end{figure}

\begin{figure}[tb]
\centering
\begin{tikzpicture}[scale=1.0]
\coordinate (orig) at (0,0);
\def\xaxissize{1.5}
\def\yaxissize{1.5}
\def\robotrotate{60}
\def\axisangle{-65}
\coordinate (fvelbase) at ($(orig)+(0+\robotrotate+\axisangle:0.85*\xaxissize)$);
\draw[->,color=blue,line width = 0.5mm] (fvelbase) -- ($(fvelbase)+(0+\robotrotate:1.0*\xaxissize)$) node[right,blue,text width = 2cm] {forward velocity $v$};
\draw[->,dashed] (fvelbase) -- ($(fvelbase)+(0+\robotrotate:0.75*\xaxissize)$) arc (0+\robotrotate:\robotrotate+\axisangle+90:0.75*\xaxissize);
\node[->,dashed] at ($(fvelbase)+(0+\robotrotate+15:0.9*\xaxissize)$) {$\localangle$};
\draw[->,orange,line width = 0.5mm] ($(orig)+(0,0) + (60+\robotrotate:1.0*\xaxissize)$) arc (60+\robotrotate:190+\robotrotate:1.0*\xaxissize) node[right,orange,text width = 2cm] {angular velocity $\omega$};
\coordinate (fsorig) at ($(orig)+(0,0)+ (\robotrotate+\axisangle:2.5*\xaxissize)$);
\draw[dashed] (orig) -- (fsorig);
\coordinate (quad) at (0,0);
\def\quadrot{15}
\def\quadbladesize{0.25*\xaxissize}
\foreach \x in {0,90,180,270}
{
\draw[thick] (quad) -- ($(quad) + (\quadrot+\x:0.25)$);
\draw[thick]  ($(quad) + (\quadrot+\x:0.25+\quadbladesize)$) circle (\quadbladesize);
}
\coordinate (cam1) at ($(quad)+(\quadrot+45:0.2)$);
\coordinate (cam2) at ($(cam1)+(\quadrot+45+30:0.3)$);
\coordinate (cam3) at ($(cam1)+(\quadrot+45-30:0.3)$);
\draw[fill = red] (cam1) -- (cam2) -- (cam3) -- (cam1);
\path[dashed,draw] (fvelbase) -- +(\robotrotate+\axisangle+90:1.0*\xaxissize);
\path[thick,draw] ($(fsorig)+0.5*(\xaxissize,0)-0.8*(0,\yaxissize)$) to [out = 120, in = \robotrotate+\axisangle-90] (fsorig) to [out =  \robotrotate+\axisangle+90, in = 180] ($(fsorig)+0.5*(\xaxissize,0)+1.0*(0,\yaxissize)$) node[below] {path};
\draw[very thick,->,red] (fsorig) -- ($(fsorig)+(\robotrotate+\axisangle+180:0.95*\xaxissize )$);
\draw[very thick,->,red] (fsorig) -- ($(fsorig)+(\robotrotate+\axisangle+90:0.95*\xaxissize )$);
\draw[very thick,<->,purple] ($(fsorig)+(\robotrotate+\axisangle-90:0.35)$) -- ($(orig)+(\robotrotate+\axisangle-90:0.35)$);
\node[purple] at ($0.5*(fsorig)+0.5*(orig)+(\robotrotate+\axisangle-90:0.6)$) {$\deviation$};
\end{tikzpicture}
\caption{\small A quadrotor with forward speed $v$, and angular velocity $\omega$. The curved black line represents a local segment of the path that the quadrotor must follow. The local Frenet-Serret frame (red) attached to the path is also shown. The quadrotor's state consists of the offset $\deviation$ and angle $\localangle$ with respect to the path.}
\label{fig:ddwmr}
\vspace{-4mm}
\end{figure}
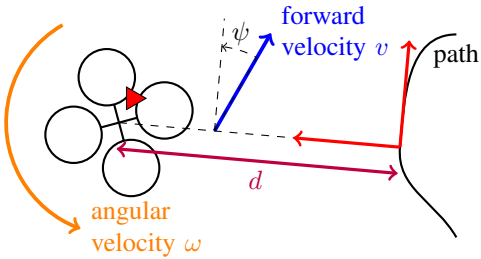 
In the rest of this section, we obtain a classifier $\classifier$
by training three one-vs-one classifiers $w^{12}$, $w^{13}$ and $w^{23}$ that distinguish between $\control_1$ and $\control_2$, $\control_1$ and $\control_3$, and $\control_2$ and $\control_3$ respectively. %
The loss function is
\vspace{-1mm}
\begin{dmath}
l_{data}(\cweights) =\lVert \cweights \rVert_2 + \gamma \sum_{k=1}^{N_\trdata} \max(0,1 - b^k \measurement^k), \label{eq:svmloss}
\end{dmath}
\vspace{-1mm}
\noindent where $\gamma > 0$ is a parameter we set as $100$. %
The loss~\eqref{eq:svmloss} implements a support vector machine. %
The class labels are then given by
\begin{dmath*}
C(\measurement) = \begin{cases}
 \control_2 & \textrm{ if } (\cweightsm^{12})^T y + \ccon^{12} <0 , (\cweightsm^{23})^T y + \ccon^{23}  > 0, \\
  \control_3 & \textrm{ if }  (\cweightsm^{13})^T y + \ccon^{13} <0 , (\cweightsm^{23})^T y + \ccon^{23}  < 0,  \\
   \control_1 & \textrm{ otherwise}.
 \end{cases}
\end{dmath*}

\subsection{Control-Oriented Training}
\label{ssec:pathtraining}
Let $\classifier_0$ be the classifier obtained when training on data without control-oriented constraints. %
We sketch the closed-loop system due to $\classifier_0$ in Figure~\ref{fig:closedloopsketch}. %
The points $x^1_e$ and $x_e^2$ are switched equilibria~\cite{Filippov1988} when the curvature of the path is strictly positive and negative respectively. %
When the curvature is zero, every point on the line $\localangle = 0$ between $x^1_e$ and $x_e^2$ is an equilibrium point. %
All trajectories either begin on this equilibrium set, or approach either $x^1_e$ or $x_e^2$. %
This analysis was presented in~\cite{PoonawalaCDC17} to explain the work in~\cite{Giusti16}. %

To demonstrate the need for control-oriented training, we mislabel the training data, and train two sets of classifiers $\classifier_1$ and $\classifier_2$ using this mislabeled data. %
Specifically, we use the data corresponding to $(\localangle^k,\deviation^k) = (\pi/6\textrm{ rad},0.5\textrm{ m}) $ as $\control_1$, instead of the data corresponding to $(\localangle^k,\deviation^k) = (0\textrm{ rad},0\textrm{ m}) $. %
We train $\classifier_1$ by using gradient-descent to minimize~\eqref{eq:svmloss} on the mislabeled data. %

We train  $\classifier_2$ so that a given point $x_e^1$, at which $\localangle = 0$ and $\deviation >0$, is a locally asymptotically stable equilibrium when the path curvature is positive. Similarly, we want a point $x_e^2$, at which $\localangle = 0$ and $\deviation <0$, to be a locally asymptotic equilibrium point when the path curvature is negative. 
We achieve this training by solving the constrained optimization formulation~\eqref{eq:intropgd1} and~\eqref{eq:intropgd2} to minimize~\eqref{eq:svmloss} on the \emph{same data} as $\classifier_1$, but subject to the following constraints. %
We constrain $w^{13}$ so that $(w_1^{13})^T \Hsetfn(x_e^1) +w_0^{13}= 0$. %
We use a Lyapunov function $V_{\Vprtition_1}(\state - x_e^1)$ to ensure that the switching between $\diffinclusion_1$ and $\diffinclusion_3$ renders $x_e^1$ to be (locally) asymptotically stable. %
The partition $\Vprtition_1$ comprises of $16$ cells depicted in Figure~\ref{fig:lyappartition}. %
The dynamics $\diffinclusion_1$, $\diffinclusion_2$, and $\diffinclusion_3$ are as in Section~\ref{ssec:modeling}, where $\rho = 1\textrm{m}$. %
We solve the projection step~\eqref{eq:intropgd2} using Algorithm~\ref{alg:projection}, with $\beta = 0.001$. %
Similarly, we train $w^{12}$ so that $x_e^2$ is a switched equilibrium and $\rho = -1\textrm{m}$, with a different Lyapunov function $V_{\Vprtition_2}(\state - x_e^2)$ as proof of local asymptotic stability. %
We train $w^{23}$ without any constraints, using loss function~\eqref{eq:svmloss}. %
\begin{figure}
\begin{tikzpicture}
\tikzstyle{myedge}=[shorten >=1pt,auto,semithick]
\tikzstyle{axes}=[->,opacity = 1.0]
\tikzstyle{mysubcaption}=[text width = 4cm, align = center]
\def \xaxissize {1.7};
\def \yaxissize {1.7};
\def \lowoffset {1.0}
\def \uppoffset {0.5}
\coordinate (or1) at (0,0);
\coordinate (lowpoint) at ($(or1)+(\lowoffset, -\yaxissize )$);
\coordinate (uppoint) at ($(or1)+(-\xaxissize,\uppoffset )$);
\coordinate (c1) at ($(or1)+(\xaxissize,\yaxissize)$);
\coordinate (c2) at ($(or1)+(-\xaxissize,\yaxissize)$);
\coordinate (c3) at ($(or1)+(-\xaxissize,-\yaxissize)$);
\coordinate (c4) at ($(or1)+(\xaxissize,-\yaxissize)$);
\coordinate (cS) at ($(or1)+(0,-\yaxissize)$);
\coordinate (cN) at ($(or1)+(0,\yaxissize)$);
\coordinate (sua) at ($(or1)+(-1,\yaxissize)$);
\coordinate (sub) at ($(or1)+(\xaxissize,-0.4)$);
\coordinate (sla) at ($(or1)+(-\xaxissize,-0.1)$);
\coordinate (slb) at ($(or1)+(0.5,-\yaxissize)$);
\path[draw, name path = AA,fill = blue!25!white] (sua) -- (sub) -- (c4) -- (slb) -- (sla) -- (c2) -- (sua);
\path[draw, name path = BB,fill = green!25!white] (sua) -- (c1) -- (sub) -- (sua);
\path[draw, name path = CC,fill = red!25!white] (sla) -- (c3) -- (slb) -- (sla);
\path[ name path = vv] ($(or1)+(-90:\yaxissize)$) -- ($(or1)+(90:\yaxissize)$);
\path[name intersections={ of=vv and AA}];
\coordinate (xe1) at (intersection-1);
\coordinate (xe2) at (intersection-2);

\draw[axes] ($(or1)-(0:\xaxissize)$) -- ($(or1)+(0:\xaxissize)$) node [anchor = north west] {$\localangle$};
\draw[axes] ($(or1)-(90:\yaxissize)$) -- ($(or1)+(90:\yaxissize)$) node [anchor = south east] {$\deviation $};
\node[below,mysubcaption] at ($(0,-\yaxissize )+(or1)$){\footnotesize a) Closed-loop system $\pwlSys_\prtition$ when $\exaenvir = 0$.};
\foreach \x in {0.9,0.7,0.5,0.3,0.1}
{
\path[ name path = temp1] ($(or1)+(-\xaxissize, \x*\yaxissize)$) -- ($(or1)+(\xaxissize, \x*\yaxissize)$);
\path[name intersections={ of=temp1 and AA}];
\coordinate (xet) at (intersection-1);
\path[ draw,->] ($(or1)+(\xaxissize, \x*\yaxissize)+ (-0.1,0)$) -- ( $(xet) + (0.1,0)$);
}
\foreach \x in {0.9,0.7,0.5,0.3,0.15}
{
\path[ name path = temp1] ($(or1)+(-\xaxissize, -\x*\yaxissize)$) -- ($(or1)+(\xaxissize, -\x*\yaxissize)$);
\path[name intersections={ of=temp1 and AA}];
\coordinate (xet) at (intersection-2);
\path[ draw,->] ($(or1)+(-\xaxissize, -\x*\yaxissize)+ (0.1,0)$) -- ( $(xet) + (-0.1,0)$);
}

\foreach \x in {0.9,0.7,0.5,0.3,0.1}
{
\path[ name path = temp1] ($(or1)+(\x*\xaxissize, \yaxissize)$) -- ($(or1)+(\x*\xaxissize, -\yaxissize)$);
\path[name intersections={ of=temp1 and AA}];
\coordinate (xet1) at (intersection-1);
\coordinate (xet2) at (intersection-2);
\path[ draw,->] ($(xet2) + (0,0.1)$) -- ( $(xet1) + (0,-0.1)$);
}
\foreach \x in {0.5,0.3,0.1}
{
\path[ name path = temp1] ($(or1)+(-\x*\xaxissize, \yaxissize)$) -- ($(or1)+(-\x*\xaxissize, -\yaxissize)$);
\path[name intersections={ of=temp1 and AA}];
\coordinate (xet1) at (intersection-1);
\coordinate (xet2) at (intersection-2);
\path[ draw,->] ($(xet1) + (0,-0.1)$) -- ( $(xet2) + (0,0.1)$);
}
\foreach \x in {0.9,0.7}
{
\path[ name path = temp1] ($(or1)+(-\x*\xaxissize, \yaxissize)$) -- ($(or1)+(-\x*\xaxissize, -\yaxissize)$);
\path[name intersections={ of=temp1 and AA}];
\coordinate (xet1) at (intersection-1);
\coordinate (xet2) at (intersection-2);
\path[ draw,->] ($(xet2) + (0,-0.1)$) -- ( $(xet1) + (0,+0.1)$);
}

\node[above right,fill = green!25!white] at (xe1) {$x_e^1$};
\path[draw,red,thick] (xe1) -- (xe2);
\draw[fill = red] (xe1) circle (0.05) ;
\node[above right,fill = blue!25!white] at (xe2) {$x_e^2$};
\draw[fill = red] (xe2) circle (0.05) ;

\node[fill = red!25!white] at ($(or1) - (1.2,1.2)$) {$\diffinclusion_2$};
\node[fill = green!25!white] at ($(or1) + (1,1)$) {$\diffinclusion_3$};
\node[fill = blue!25!white] at ($(or1) + 1.1*(1.1,-1)$) {$\diffinclusion_1$};
\coordinate (or1) at (4.2,0);
\coordinate (lowpoint) at ($(or1)+(\lowoffset, -\yaxissize )$);
\coordinate (uppoint) at ($(or1)+(-\xaxissize,\uppoffset )$);
\coordinate (c1) at ($(or1)+(\xaxissize,\yaxissize)$);
\coordinate (c2) at ($(or1)+(-\xaxissize,\yaxissize)$);
\coordinate (c3) at ($(or1)+(-\xaxissize,-\yaxissize)$);
\coordinate (c4) at ($(or1)+(\xaxissize,-\yaxissize)$);
\coordinate (cS) at ($(or1)+(0,-\yaxissize)$);
\coordinate (cN) at ($(or1)+(0,\yaxissize)$);
\coordinate (sua) at ($(or1)+(-1,\yaxissize)$);
\coordinate (sub) at ($(or1)+(\xaxissize,-0.4)$);
\coordinate (sla) at ($(or1)+(-\xaxissize,-0.1)$);
\coordinate (slb) at ($(or1)+(0.5,-\yaxissize)$);
\path[draw, name path = AA,fill = blue!25!white] (sua) -- (sub) -- (c4) -- (slb) -- (sla) -- (c2) -- (sua);
\foreach \x in {0.9,0.7,0.5,0.3,0.15}
{
\path[ name path = temp1,draw] ($(or1)+(-\x*\xaxissize,\yaxissize)$) to [out = -85,in = 180] ($(or1)+(0,\yaxissize-2*\x*\yaxissize)$) to [out = 0,in = -95] ($(or1)+(\x*\xaxissize,\yaxissize)$);
\path[ name path = temp1,draw,->] ($(or1)+(-\x*\xaxissize,\yaxissize)$) to [out = -85,in = 180] ($(or1)+(0,\yaxissize-2*\x*\yaxissize)$);

}

\path[draw, name path = BB,fill = green!25!white] (sua) -- (c1) -- (sub) -- (sua);
\path[draw, name path = CC,fill = red!25!white] (sla) -- (c3) -- (slb) -- (sla);
\path[ name path = vv] ($(or1)+(-90:\yaxissize)$) -- ($(or1)+(90:\yaxissize)$);
\path[name intersections={ of=vv and AA}];
\coordinate (xe1) at (intersection-1);
\coordinate (xe2) at (intersection-2);

\draw[axes] ($(or1)-(0:\xaxissize)$) -- ($(or1)+(0:\xaxissize)$) node [anchor = north west] {$\localangle$};
\draw[axes] ($(or1)-(90:\yaxissize)$) -- ($(or1)+(90:\yaxissize)$) node [anchor = south east] {$\deviation$};
\node[below,mysubcaption] at ($(0,-\yaxissize )+(or1)$){\footnotesize b) Closed-loop system $\pwlSys_\prtition$ when $\exaenvir > 0$.};

\foreach \x in {0.9,0.7,0.5,0.3,0.1}
{
\path[ name path = temp1] ($(or1)+(-\xaxissize, \x*\yaxissize)$) -- ($(or1)+(\xaxissize, \x*\yaxissize)$);
\path[name intersections={ of=temp1 and AA}];
\coordinate (xet) at (intersection-1);
\path[ draw,->] ($(or1)+(\xaxissize, \x*\yaxissize)+ (-0.1,0)$) -- ( $(xet) + (0.1,0)$);
}
\foreach \x in {0.9,0.7,0.5,0.3,0.15}
{
\path[ name path = temp1] ($(or1)+(-\xaxissize, -\x*\yaxissize)$) -- ($(or1)+(\xaxissize, -\x*\yaxissize)$);
\path[name intersections={ of=temp1 and AA}];
\coordinate (xet) at (intersection-2);
\path[ draw,->] ($(or1)+(-\xaxissize, -\x*\yaxissize)+ (0.1,0)$) -- ( $(xet) + (-0.1,0)$);
}

\node[above right,fill = green!25!white] at (xe1) {$x_e^1$};
\draw[fill = red] (xe1) circle (0.05) ;

\node[fill = red!25!white] at ($(or1) - (1.2,1.2)$) {$\diffinclusion_2$};
\node[fill = green!25!white] at ($(or1) + (1,1)$) {$\diffinclusion_3$};
\node[fill = blue!25!white] at ($(or1) + 1.1*(1.1,-1)$) {$\diffinclusion_1$};	

\end{tikzpicture}
\caption{\small Sketch of closed-loop dynamics due to classifier $\classifier_0$ for path following for different curvatures. Points $x_e^1$ and $x_e^2$ are switched equilibria~\cite{Filippov1988} when curvature $\rho$ is positive and negative respectively. Their convex hull forms a line of equilibria when $\rho = 0$. }
\label{fig:closedloopsketch}
\vspace{-4mm}
\end{figure}
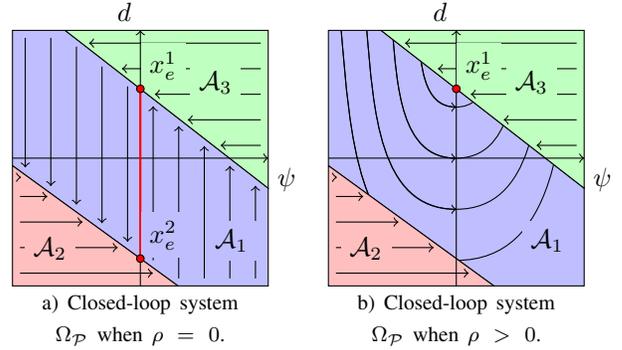
\vspace{-2mm}
\subsection{Results}
We simulate the path-following control of a quadrotor when using $\classifier_1$ and $\classifier_2$ (separately). %
Figure~\ref{fig:stdsvmstraight} shows the resulting trajectories, in local coordinates. %
We see that for the classifier $\classifier_2$ trained with control constraints, the trajectories reach the set of equilibria points between $x_e^1$ and $x_e^2$. %
The switching surfaces are similar to those in Figure~\ref{fig:closedloopsketch}a.
For the classifier $\classifier_1$ trained without constraints on $\cweights$, some trajectories move away from the origin, in fact the quadrotor crashes in the simulation. %
In the remaining trajectories the quadrotor is reaches a switching surface between turning left and turning right, and consequently oscillates between the two without moving along the path. %
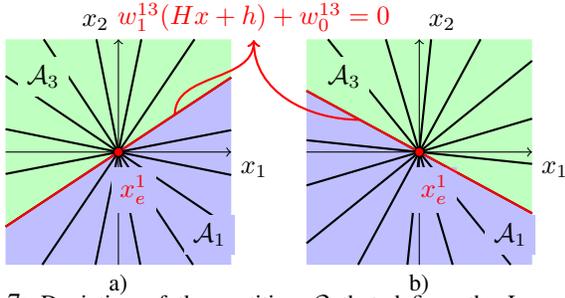
\begin{figure}[tb]
\centering
\begin{tikzpicture}
\tikzstyle{axes}=[->,opacity = 1.0]
\tikzstyle{mysubcaption}=[text width = 4cm, align = center]
\def \xaxissize {1.5};
\def \yaxissize {1.5};
\coordinate (or1) at (0,0);
\coordinate (c1) at ($(or1)+(\xaxissize,\yaxissize)$);
\coordinate (c2) at ($(or1)+(-\xaxissize,\yaxissize)$);
\coordinate (c3) at ($(or1)+(-\xaxissize,-\yaxissize)$);
\coordinate (c4) at ($(or1)+(\xaxissize,-\yaxissize)$);
\coordinate (cS) at ($(or1)+(0,-\yaxissize)$);
\coordinate (cN) at ($(or1)+(0,\yaxissize)$);
\path[name path = AA] (c1) -- (c2) -- (c3) -- (c4) -- (c1);
\def\classangle{11}
\path[name path = ray] (or1) -- ($(or1)+(1*22.5+\classangle:1.2*\xaxissize)$);
\path[name intersections={ of=ray and AA}];
\coordinate (xeta) at (intersection-1);
\path[name path = ray] (or1) -- ($(or1)+(9*22.5+\classangle:1.2*\xaxissize)$);
\path[name intersections={ of=ray and AA}];
\coordinate (xeta2) at (intersection-1);
\fill[green!25!white] (xeta) -- (c1) -- (c2) -- (xeta2) -- (xeta);
\fill[blue!25!white] (xeta) -- (c4) -- (c3) -- (xeta2) -- (xeta);
\draw[axes] ($(or1)-(0:\xaxissize)$) -- ($(or1)+(0:\xaxissize)$) node [anchor = north west] {$\state_1$};
\draw[axes] ($(or1)-(90:\yaxissize)$) -- ($(or1)+(90:\yaxissize)$) node [anchor = south east] {$\state_2 $};
\foreach \x in{1,...,16}
{
\path[name path = ray] (or1) -- ($(or1)+(\x*22.5+\classangle:1.3*\xaxissize)$);
\path[name intersections={ of=ray and AA}];
\coordinate (xet) at (intersection-1);
\path[name path = ray] (or1) -- ($(or1)+(\x*22.5+22.5+\classangle:1.3*\xaxissize)$);
\path[draw,thick] (or1) -- (xet);
}
\node[fill = green!25!white] at ($(or1) + (-1,1)$) {$\diffinclusion_3$};
\node[fill = blue!25!white] at ($(or1) + 1.1*(1.1,-1)$) {$\diffinclusion_1$};
\node[red,fill = blue!25!white] at ($(or1) + (0.2,-0.5)$) {$x_e^1$};
\draw[red,thick] (xeta) -- (xeta2);
\draw[fill=red] (or1) circle (0.06);
\node[red] (surfname) at ($(or1) + (1.2*\xaxissize,1.2*\yaxissize)$) {$w^{13}_1( H x +h) + w_0^{13} = 0$};
\path[draw,thick,red,->] ($(or1)+(22.5+\classangle:0.6*\xaxissize)$) to [out=90,in=-90] (surfname);
\node[below, mysubcaption] at ($(0,-\yaxissize )+(or1)$){\small a)};

\def \xaxissize {1.5};
\def \yaxissize {1.5};
\coordinate (or1) at (4,0);
\coordinate (c1) at ($(or1)+(\xaxissize,\yaxissize)$);
\coordinate (c2) at ($(or1)+(-\xaxissize,\yaxissize)$);
\coordinate (c3) at ($(or1)+(-\xaxissize,-\yaxissize)$);
\coordinate (c4) at ($(or1)+(\xaxissize,-\yaxissize)$);
\coordinate (cS) at ($(or1)+(0,-\yaxissize)$);
\coordinate (cN) at ($(or1)+(0,\yaxissize)$);
\path[name path = AA] (c1) -- (c2) -- (c3) -- (c4) -- (c1);
\def\classangle{-51}
\path[name path = ray] (or1) -- ($(or1)+(1*22.5+\classangle:1.3*\xaxissize)$);
\path[name intersections={ of=ray and AA}];
\coordinate (xeta) at (intersection-1);
\path[name path = ray] (or1) -- ($(or1)+(9*22.5+\classangle:1.3*\xaxissize)$);
\path[name intersections={ of=ray and AA}];
\coordinate (xeta2) at (intersection-1);
\fill[green!25!white] (xeta) -- (c1) -- (c2) -- (xeta2) -- (xeta);
\fill[blue!25!white] (xeta) -- (c4) -- (c3) -- (xeta2) -- (xeta);
\draw[axes] ($(or1)-(0:\xaxissize)$) -- ($(or1)+(0:\xaxissize)$) node [anchor = north west] {$\state_1$};
\draw[axes] ($(or1)-(90:\yaxissize)$) -- ($(or1)+(90:\yaxissize)$) node [anchor = south west] {$\state_2 $};
\foreach \x in{1,...,16}
{
\path[name path = ray] (or1) -- ($(or1)+(\x*22.5+\classangle:1.3*\xaxissize)$);
\path[name intersections={ of=ray and AA}];
\coordinate (xet) at (intersection-1);
\path[name path = ray] (or1) -- ($(or1)+(\x*22.5+22.5+\classangle:1.3*\xaxissize)$);
\path[draw,thick] (or1) -- (xet);
}
\node[fill = green!25!white] at ($(or1) + (-1,1)$) {$\diffinclusion_3$};
\node[fill = blue!25!white] at ($(or1) + 1.1*(1.1,-1)$) {$\diffinclusion_1$};
\node[red,fill = blue!25!white] at ($(or1) + (0.2,-0.5)$) {$x_e^1$};
\draw[red,thick] (xeta) -- (xeta2);
\draw[fill=red] (or1) circle (0.06);
\path[draw,thick,red,->] ($(or1)+(22.5+\classangle:-0.6*\xaxissize)$) to [out=180,in=-90] (surfname);
\node[below, mysubcaption] at ($(0,-\yaxissize )+(or1)$){\small b)};
\end{tikzpicture}
\vspace{-4mm}
\caption{\small Depiction of the partition $\Vprtition$ that defines the Lyapunov function $V_{\Vprtition}(\state)$ used to train $\cweights^{13} = (\cweights_1^{13},\cweights^{13}_0)$. Note that $\prtition = \Vprtition$, and $\Vprtition$ depends on the classifier parameters. The point $x_e^1$ is not asymptotically stable for the switching dynamics in a), but it is for that in b). }
\label{fig:lyappartition}
\vspace{-2mm}
\end{figure}
\begin{figure}
\centering
\begin{subfigure}{0.24\textwidth}
\centering
    \includegraphics[width=0.99\linewidth]{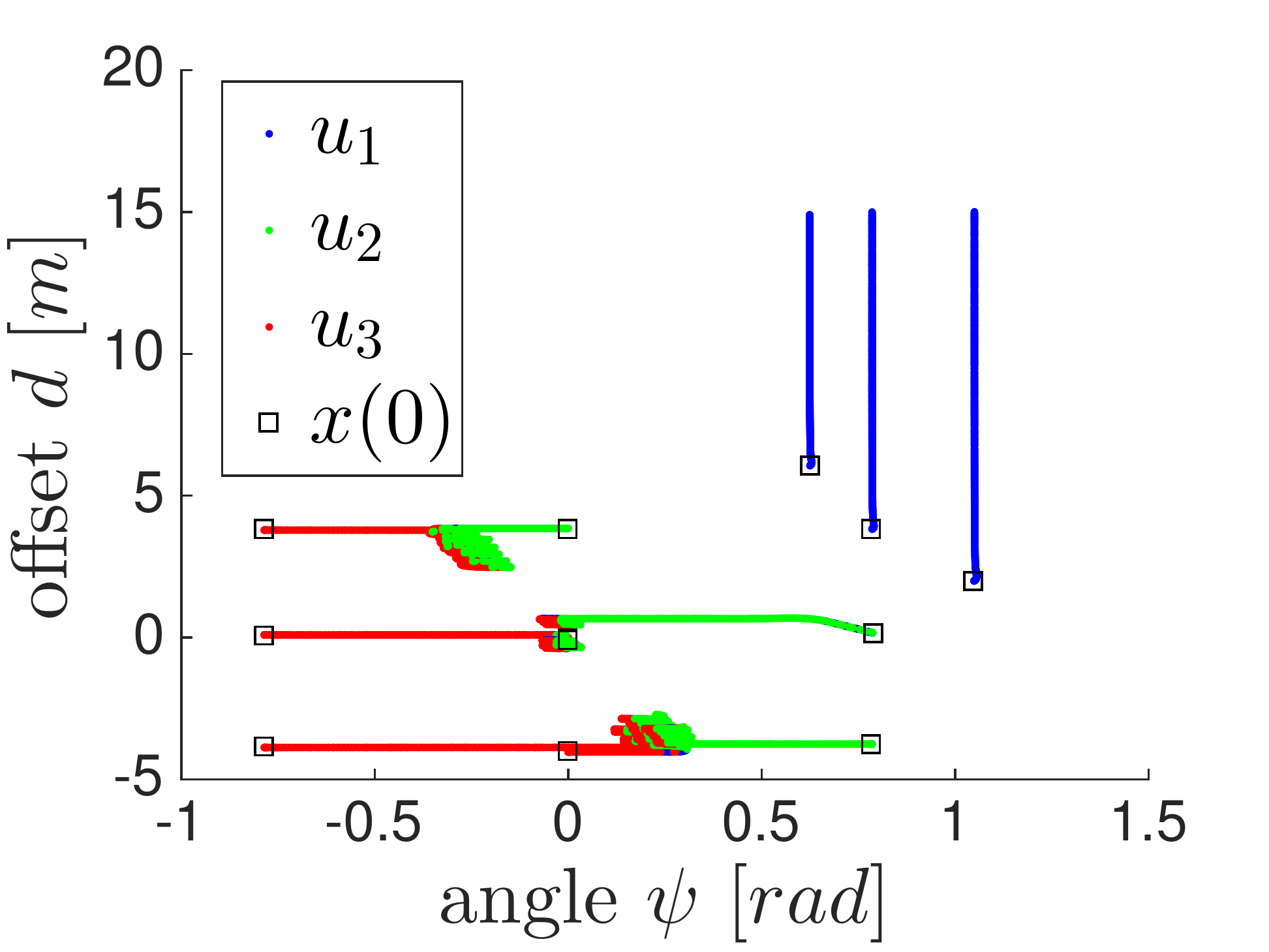}
    \caption{Classifier $\classifier_1$}
    \label{fig:traincanyon}
\end{subfigure}%
\begin{subfigure}{0.24\textwidth}
\centering
    \includegraphics[width=0.99\linewidth]{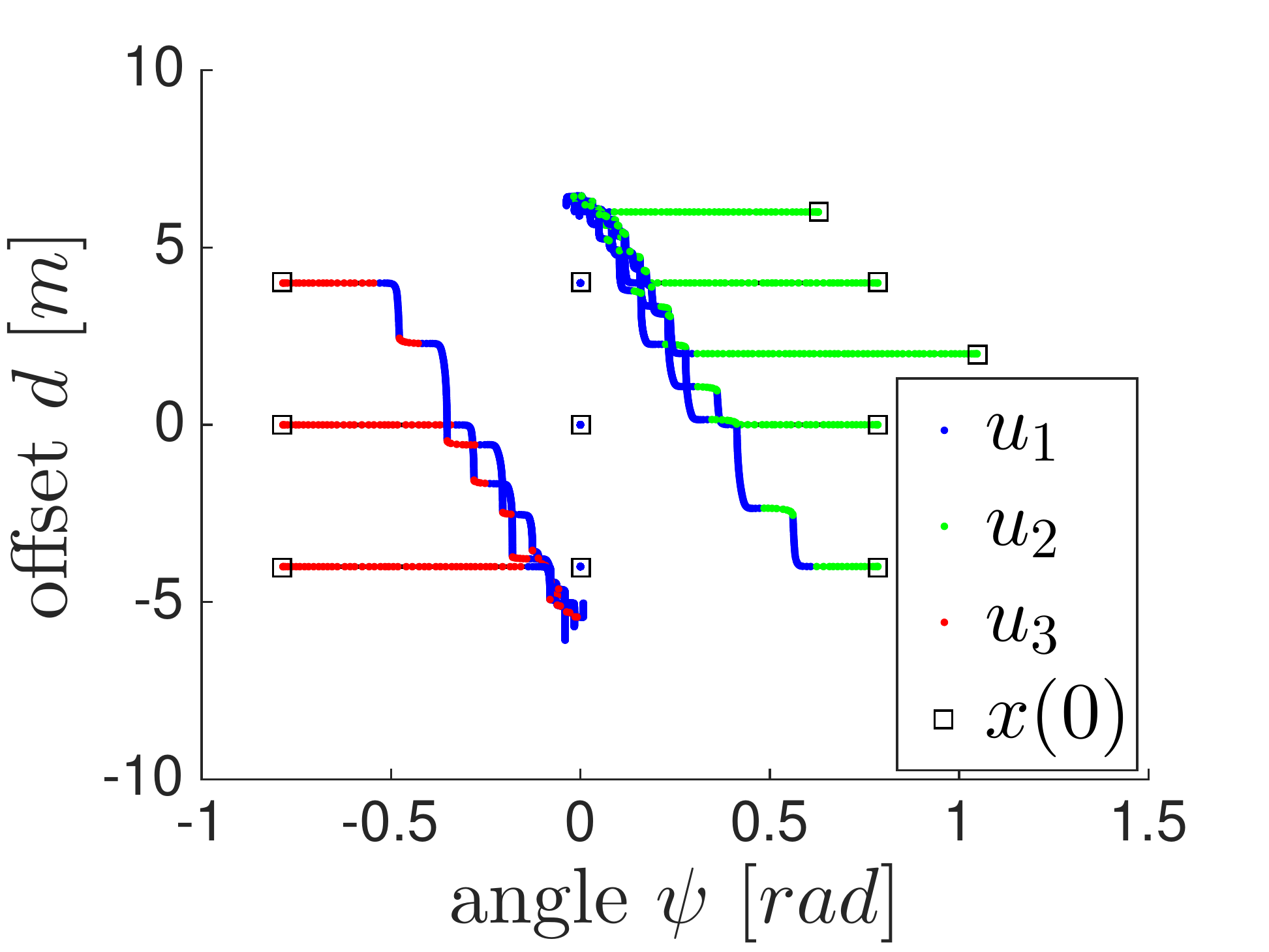}
    \caption{Classifier $\classifier_2$}
    \label{fig:testcanyon}
\end{subfigure}
\caption{\small Trajectories of the quadrotor, in local coordinates, when following a path using classifiers trained with the same data but different methods (see Section~\ref{ssec:pathtraining}). a) The trajectories due to classifier $\classifier_1$, trained without control considerations, either exhibit oscillations due to switching between turning right and left, or the quadrotor moves away from the center of the path towards the sides, eventually crashing. b) The trajectories due to classifier $\classifier_2$, trained with control constraints, approach the set of equilibria between the two points $x_e^1$ and $x_e^2$ (see Figure~\ref{fig:closedloopsketch}a) which exist by design of the classifier (see Section~\ref{ssec:pathtraining}).}\label{fig:stdsvmstraight}
\end{figure}
%
\vspace{-1mm}
\section{Conclusions And Future Work}
We have presented a novel training algorithm for classifiers that incorporate control-oriented constraints on the classifier parameters. %
We derived these constraints by modeling the closed-loop system as a piece-wise affine differential inclusion, and using polyhedral Lyapunov functions to verify desired closed-loop properties. %
We show the usefulness of this novel training method in a simulation of a quadrotor navigating terrain by classifying high-dimensional sensor measurements into one of three possible velocities. %

While we have demonstrated the value of the proposed training method through our case study, the method presents some issues to be addressed. %
Our method requires us to derive a piecewise affine differential inclusion that over-approximates the effect of the classifier-in-the-loop architecture. %
We do not provide a systematic method to derive the tightest possible over-approximation with respect to the control properties of interest. %
It is possible that the over-approximation we construct does not satisfy the control properties, even though a tighter one exists that would satisfy them. %
Furthermore, it is unclear how the choice of the partitions for the differential inclusion and the polyhedral Lyapunov function affects the convergence of Algorithm~\ref{alg:projection}. %
We investigate these issues in future work. %

In our framework, we have assumed that the set of controls $\Control$ is known, via prior knowledge or human expertise, and the measurements are labeled. %
In future work, we will investigate the case where the control set $\Control$ needs to be determined, and/or the measurements are unlabeled. %
\bibliographystyle{IEEEtran}
\bibliography{ms}
\nocite{Wicks94}
\appendix
\def\localVf{V}
\def\aDomain{D}
\def\aEq{x_e}

We consider a differential inclusion of the form 
\begin{dmath}
\dot{x} \in \diffinclusion(x) \label{eq:appdiffinc}	
\end{dmath}
\noindent and a candidate Lyapunov function $\localVf(x)$. %
We will recount results from~\cite{Cortes2008} that provide conditions under which $\localVf(x)$ allows one to claim different properties of~\eqref{eq:appdiffinc}. %

\begin{defn}[Locally Lipschitz]
	A function $\localVf \colon \R^n \to \R$ is \emph{locally Lipschitz} at $x \in \R^n$ if there exists $\gamma, \epsilon \in \R_{+}$ such that $\lVert \localVf(y) - \localVf(x) \rVert \leq \gamma \lVert x - y \rVert$ for all $y \in B_{\epsilon}(x)$.
\end{defn}

\begin{defn}[Generalized gradient]
Let $\localVf$ be a locally Lipschitz function, and let $Z$ be the set of points where $\localVf$ fails to be differentiable. %
The generalized gradient $\partial \localVf(x)$ at $x$ is defined by
\begin{dmath}
	\partial \localVf (x)= co \{\lim_{i \to \infty} \nabla \localVf(x) \colon x_i \hiderel{\to} x, x \hiderel{\notin} S \cup Z \},
\end{dmath}
\noindent where $S$ is any set of measure zero that can be arbitrarily chosen to simplify the computation. %
The resulting set $\partial \localVf (x)$ is independent of the choice of $S$.
\end{defn}
\vspace{-1mm}
\begin{defn}[Set-valued Lie Derivative]
Given a locally Lipschitz function $\localVf : \R^n \to \R$ and a set-valued map $\diffinclusion \colon \R^n \to 2^{\R^n}$, the set-valued Lie derivative $\mathcal{L}_{\diffinclusion}f(x)$ of $\localVf$ with respect to $\diffinclusion$ at $x \in \R^n$ is given by
\begin{dmath*}
\mathcal{L}_{\diffinclusion}\localVf(x) = \{ z \hiderel{\in} \R 	\colon \exists v \hiderel{\in} \diffinclusion \textrm{ such that } \zeta^T v \hiderel{=} a\ \forall \zeta \hiderel{\in} \partial \localVf(x) \}.
\end{dmath*}
\end{defn}

\begin{defn}[Caratheodory solution]
	A Caratheodory solution of $\dot x(t) \in \diffinclusion(x)$ defined on $[t_0, t_1] \subset [0,\infty)$ is an absolutely continuous map $
x \colon [t_0, t_1] \to \R^n$ such that $\dot x (t) \in \diffinclusion(x)$ for almost every $t \in [t_0, t_1]$.
\end{defn}

Note that Caratheodory solutions of differential inclusions are identical to Filippov solutions of discontinuous systems with the typical convex relaxation of the dynamics. %
We are now ready to present three results from~\cite{Cortes2008}. %
\vspace{-1mm}
\begin{prop}[Proposition S1 from~\cite{Cortes2008}]
\label{prop:S1}
Let $\diffinclusion$ be locally bounded and take nonempty, compact, and convex values. Assume that the set-valued map $x \mapsto \diffinclusion(x)$ is upper
semicontinuous. Then, for all $x_0 \in \R^n$ there exists a Caratheodory solution of $\dot x(t) \in \diffinclusion(x)$ with initial condition $x(t_0) = x_0$.
\end{prop}
\vspace{-1mm}
\begin{thm}[Theorem~1 in~\cite{Cortes2008}]
\label{thm:cortes1}
Let $\diffinclusion \colon \R^n \mapsto 2^{\R^n}$ be a set-valued map satisfying the hypotheses of
Proposition~\ref{prop:S1}, let $\aEq$ be an equilibrium of the differential inclusion~\eqref{eq:appdiffinc}, and let $\aDomain \subseteq \R^n$ be an
open and connected set with $\aEq \in \aDomain$. Furthermore, let $\localVf : \R^n \mapsto \R$ be such that the following
conditions hold:
\begin{itemize}
	\item[(i)] $\localVf$ is locally Lipschitz and regular on $\aDomain$.
	\item[(ii)] $\localVf(\aEq) = 0$, and $\localVf(x) > 0$ for $x \in \aDomain \backslash {\aEq}$.
	\item[(iii)] $\max \mathcal{L}_{\diffinclusion}\localVf(x) \leq 0$ for each $x \in \aDomain$.
\end{itemize}
Then, $\aEq$ is a stable equilibrium of~\eqref{eq:appdiffinc}. %
In addition, if $(iii)$ above is replaced by
\begin{itemize}
	\item[(iii)']  $\max \mathcal{L}_{\diffinclusion}\localVf(x) < 0$ for each $x \in \aDomain\backslash {\aEq}$,
\end{itemize}
then $\aEq$ is an asymptotically stable equilibrium of~\eqref{eq:appdiffinc}.	
\end{thm}
\vspace{-1mm}
\begin{thm}[Theorem~2 in~\cite{Cortes2008}]
\label{thm:cortes2}
	Let $\diffinclusion \colon \R^n \mapsto 2^{\R^n}$ be a set-valued map satisfying the hypotheses of
Proposition~\ref{prop:S1}, and let $\localVf : \R^n \mapsto \R$ be a locally Lipschitz and regular function. Let $S \subset \R^n$ be compact and strongly invariant for~\eqref{eq:appdiffinc}, and assume that $\max \mathcal{L}_{\diffinclusion}\localVf(x) \leq 0$ for each $x \in S$. Then,
all solutions $x : [0,\infty) \mapsto \R^n$ of~\eqref{eq:appdiffinc} starting at $S$ converge to the largest weakly invariant set $M$
contained in
	\begin{dmath}
	S \cap \{x \in \R^n \colon 0 \in  \mathcal{L}_{\diffinclusion}\localVf(x) \}	
	\end{dmath}
Moreover, if the set $M$ consists of a finite number of points, then the limit of each solution starting
in $S$ exists and is an element of $M$.
\end{thm}

\vspace{-1mm}
The results above consider general differential inclusions and candidate Lyapunov functions. %
These conditions involve quantifiers of the form `for each $x \in S$' where $S$ is a set. %
For the piece-wise affine functions we consider, we can remove these quantifiers using the following result. %
Let $x \gneq 0 \iff x \geq 0, x \neq 0$. %
\vspace{-1mm}
\begin{lem}[Lemma $4.7$~\cite{JohanssonThesis}]
\label{lem:poslyap}
The following are equivalent
\begin{enumerate}
\item $\vMatrix \state \gneq 0 \implies p^T  \state > 0$.
\item there exists $\mu \in \R^{n}$ such that 
	\begin{align}
			\vMatrix^T \mu   &= p, \textrm{ and }\label{eq:toalyapcond1}\\
			\mu &> 0. \label{eq:toalyapcond1b}
	\end{align}
\end{enumerate}
\end{lem}
The result establishes equivalence between the two statements. %
When the functions are affine instead of linear, we can derive a similar condition but the implication flows only in one direction. %
The formal statement is given below. %
\vspace{-1mm}
\begin{lem}
\label{lem:mainlemma}
	Let the set $\{x \in \R^n \vert E x + e\gneq 0\}$ be non-empty. 
	If there exists $\nu \in \R^{n+1}$ such that 
	\begin{align}
			\bmat{E & e\\ 0 &1}^T \nu + \bmat{A & a}^T p &= 0, \textrm{ and }\label{eq:toalyapcond3}\\
			\nu &> 0, \label{eq:toalyapcond4}
	\end{align}
	then $E \state + e\gneq 0 \implies p^T \left(A \state +a \right) < 0$. %
\end{lem}
\begin{proof}
	Redefine $\vMatrix$ to be $\bmat{E & e\\ 0 &1}$ and $p$ to be $-\bmat{A & a}^T p$ in Lemma~\ref{lem:poslyap}. %
	By Lemma~\ref{lem:poslyap}, if condition~\eqref{eq:toalyapcond3} holds, then $E \state + e z\gneq 0 \implies p^T A \state + p^T a z  < 0$. Setting $z=1$, completes the proof. 
\end{proof}
\vspace{-1mm}
\begin{lem}
\label{lem:sublemma}
Let $\State_i \cap \State_j =\{ x \in \R^n \vert \vcVec_{ij}^T x =0 \}$. %
If $\exists \lambda > 0$ such that 
\begin{dmath}
p_i - p_j = \lambda \vcVec_{ij}, 
\end{dmath}
\noindent then	$(p_i - p_j)^T  \state = 0\ \forall \state\in \State_i \cap \State_j$.
\end{lem}
\begin{proof}
We omit this proof due to its similarity to Lemma~\ref{lem:mainlemma}. %
\end{proof}
\vspace{-1mm}
We are now ready to prove Lemma~\ref{lem:mainconstraintlemma}. 
\begin{proof}
We first show that the Lyapunov function $V_\Vprtition(x)$ satisfying constraints~\eqref{eq:lemconstraintsfirst}, \eqref{eq:lemconstraintssecond}, \eqref{eq:lemconstraintscont1}, and~\eqref{eq:lemconstraintslast} possess the desired properties of Theorem~\ref{thm:cortes1}. %
By construction, $V_\Vprtition(x)$ is piece-wise linear on $\R^n$ except at the boundaries between cells of $\Vprtition$. %
If constraints~\eqref{eq:lemconstraintscont1}, and~\eqref{eq:lemconstraintslast} are satisfied, then by Lemma~\ref{lem:sublemma} $V_\Vprtition(x)$ is single-valued at these boundaries, so that $V_\Vprtition(x)$ is a continuous function on $\R^n$. %
By Lemma~\ref{lem:poslyap}, if conditions~\eqref{eq:lemconstraintsfirst}, \eqref{eq:lemconstraintssecond} are satisfied, then $x \in \qState_j, x \neq 0 \implies p_j^T x = V_\Vprtition(x) > 0$. %
By construction, $V_\Vprtition(0) = 0$. %
The definition of the cells $\qState_j \in \Vprtition$ imply that $V_{\Vprtition}(x)$ is a convex function of the form $V_{\Vprtition}(\state) = \max_{j \in \Vindex} p_j^T x$. %
Convex functions are regular~\cite{Cortes2008}. %
Therefore, if constraints~\eqref{eq:lemconstraintsfirst}, \eqref{eq:lemconstraintssecond}, \eqref{eq:lemconstraintscont1}, and~\eqref{eq:lemconstraintslast} are satisfied, then the candidate Lyapunov function $V_{\Vprtition}(x)$ satisfies the conditions of Theorem~\ref{thm:cortes1}. %

For system $\pwlSys_\prtition$, the differential inclusion is a piece-wise convex combination of a finite number of affine functions, and therefore Proposition~\ref{prop:S1} holds. %
The set-valued Lie derivative for $\pwlSys_\prtition$ reduces to
\begin{dmath*}
	\mathcal{L}_{\diffinclusion} V_{\Vprtition}(x) = co_{k \in \Dindex{i}} \left( p_j^T (A_{ik} x+ a_{ik}) \right) \condition{if $x \in \pState_i \cap \qState_j$}.
\end{dmath*}

If conditions~\eqref{eq:lemconstraintsdecrease} and~\eqref{eq:lemconstraintsdecrease2} hold, then by Lemma~\ref{lem:mainlemma} for each $x \in \pState_i \cap \qState_j$ and $k \in \Dindex{i}$, $E_i x + e_i \gneq 0 \implies p_j^T (A_{ik} x+ a_{ik}) < 0$. %
In other words, conditions $(iii)$ and $(iii)'$ hold for all $x \in Dom(\prtition)$. %
Therefore,$V_{\Vprtition}(t)$ value decreases along all solutions of $\pwlSys_{\prtition}$. %
Since $V_{\Vprtition}$ is continuous and positive definite, conclusion $1)$  of Lemma~\ref{lem:mainconstraintlemma} immediately follows. %
If $x_e = 0 \in Dom(\prtition)$, Theorem~\ref{thm:cortes1} implies that the origin is asymptotically stable. %
Finally, condition $2)$ of Lemma~\ref{lem:mainconstraintlemma} follows by considering $S_{max}$ as the strongly invariant set $S$ and $S_{min}$ as the weakly invariant set $M$ in Theorem~\ref{thm:cortes2}. %
\end{proof}

\end{document}